\documentclass{article}

\usepackage{lineno,hyperref}
\usepackage{amsmath}
\usepackage{amssymb}
\usepackage{amsthm}
\modulolinenumbers[5]
\usepackage{enumitem}
\usepackage[all]{xy}
\SelectTips{cm}{12}

\newtheorem{lemma}{Lemma}
\newtheorem*{lemma*}{Lemma}
\newtheorem{prop}{Proposition}
\newtheorem{theorem}{Theorem}
\newtheorem*{theorem*}{Theorem}

\newcommand{\dotminus}{\mathbin{\dot{-}}}

\newcommand{\dotoplus}{\mathbin{\dot{\oplus}}}

\newcommand{\sMat}[4]{\bigl(\begin{smallmatrix}#1\mathstrut&#2\mathstrut\\#3\mathstrut&#4\mathstrut\end{smallmatrix}\bigr)}
\newcommand{\sCol}[2]{\bigl(\begin{smallmatrix}#1\mathstrut\\#2\mathstrut\end{smallmatrix}\bigr)}

\DeclareMathOperator\mat{M}
\DeclareMathOperator\unit{U}

\DeclareMathOperator\stunit{StU}

\DeclareMathOperator\stmap{st}

\DeclareMathOperator\diag{D}

\DeclareMathOperator\glin{GL}

\DeclareMathOperator\stlin{St}
\DeclareMathOperator\symp{Sp}

\DeclareMathOperator\sorth{SO}

\DeclareMathOperator\group{G}

\DeclareMathOperator\ofasymp{ASp}
\DeclareMathOperator\ofaorth{AO}

\DeclareMathOperator\Ker{Ker}

\newcommand{\leqt}{\trianglelefteq}
\newcommand{\inv}[1]{\!\;\overline{\!\!\:#1\vphantom !\!\!\:}\;\!}

\DeclareMathOperator{\id}{id}

\DeclareMathOperator{\Aut}{Aut}
\DeclareMathOperator{\Heis}{Heis}

\newcommand{\up}[2]{{^{#1}\!{#2}}}

\title{A presentation of relative unitary Steinberg groups}

\author{
  Egor Voronetsky\thanks{Research is supported by the Russian Science Foundation grant 22-21-00257.} \\
  Chebyshev Laboratory, \\
  St. Petersburg State University, \\
  14th Line V.O., 29B, \\
  Saint Petersburg 199178 Russia \\
}

\begin{document}
\maketitle

\begin{abstract}
We find an explicit presentation of relative odd unitary Steinberg groups constructed by odd form rings and of relative doubly laced Steinberg groups over commutative rings, i.e. the Steinberg groups associated with the Chevalley group schemes of the types \(\mathsf B_\ell\), \(\mathsf C_\ell\), \(\mathsf F_4\) for \(\ell \geq 3\). For simply laced root systems such result is already known.
\end{abstract}


\section{Introduction}

Relative Steinberg groups \(\stlin(R, I)\) were defined in the stable linear case by F. Keune and J.-L. Loday in \cite{Keune, Loday}. Namely, this group is just
\[\frac{\Ker\bigl(p_{2*} \colon \stlin(I \rtimes R) \to \stlin(R)\bigr)}{\bigl[\Ker(p_{1*}), \Ker(p_{2*})\bigr]},\]
where \(R\) is a unital associative ring, \(I \leqt R\), \(p_1 \colon I \rtimes R \to R, a \rtimes p \mapsto a + p\) and \(p_2 \colon I \rtimes R \to R, a \rtimes p \mapsto p\) are ring homomorphisms. Such a group is a crossed module over \(\stlin(R)\) generated by \(x_{ij}(a)\) for \(a \in I\) with the ``obvious'' relations that are satisfied by the generators \(t_{ij}(a)\) of the normal subgroup \(\mathrm E(R, I) \leqt \mathrm E(R)\). It is classically known that \(\stlin(R, I)\) is generated by \(z_{ij}(a, p) = \up{x_{ji}(p)}{x_{ij}(a)}\) for \(a \in I\) and \(p \in R\) as an abstract group, the same holds for the relative elementary groups.

Such relative Steinberg groups and their generalizations for unstable linear groups and Chevalley groups are used in, e.g., proving centrality of \(\mathrm K_2\) \cite{CentralityD, CentralityE}, a suitable local-global principle for Steinberg groups \cite{LocalGlobalC, CentralityD, Tulenbaev}, early stability of \(\mathrm K_2\) \cite{Tulenbaev}, and \(\mathbb A^1\)-invariance of \(\mathrm K_2\) \cite{Horrocks, AInvariance}. In \cite[theorem 9]{CentralityE} S. Sinchuk proved that all relations between the generators \(z_\alpha(a, p)\) of \(\stlin(\Phi; R, I)\), where \(\Phi\) is a root system of type \(\mathsf{ADE}\) and \(R\) is commutative, come from various \(\stlin(\Psi; R, I)\) for root subsystems \(\Psi \subseteq \Phi\) of type \(\mathsf A_3\), i.e. \(\stlin(\Phi; R, I)\) is the amalgam of \(\stlin(\Psi; R, I)\) with identified generators \(z_\alpha(a, p)\).

There exist explicit presentations (in the sense of abstract groups) of relative unstable linear and symplectic Steinberg groups in terms of van der Kallen's generators, i.e. analogues of arbitrary transvections in \(\glin(n, R)\) or \(\symp(2n, R)\), see \cite{RelativeC, Tulenbaev} and \cite[proposition 8]{CentralityE}.

In \cite{RelStLin} we determined the relations between the generators \(z_\alpha(a, p)\) in the following two cases:
\begin{itemize}
\item for relative unstable linear Steinberg groups \(\stlin(n; R, I)\) with \(n \geq 4\),
\item for relative simply laced Steinberg groups \(\stlin(\Phi; R, I)\) with \(\Phi\) of rank \(\geq 3\).
\end{itemize}
In turns out that all the relations between \(z_\alpha(a, p)\) come from \(\stlin(\Psi; R, I)\) for \(\Psi\) of types \(\mathsf A_2\) and \(\mathsf A_1 \times \mathsf A_1\), thus Sinchuk's result may be strengthened a bit.

The relations for the simply laced Steinberg groups are easily obtained from the linear case and Sinchuk's result. In the linear case we actually considered Steinberg groups associated with a generalized matrix ring \(T\) instead of \(\mat(n, R)\), i.e. if \(T\) is a ring with a complete family of \(n\) full orthogonal idempotents. Such a generality is convenient to apply ``root elimination'', i.e. replacing the generators of a Steinberg group parameterized by a root system \(\Phi\) by some new generators parameterized by a system of relative roots \(\Phi / \alpha\). Moreover, instead of an ideal \(I \leqt R\) we considered arbitrary crossed module \(\delta \colon A \to T\) in the sense of associative rings since this is necessary for, e.g., applying the method of Steinberg pro-groups \cite{CentralityBC, LinK2}.

In this paper we find relations between the generators \(z_\alpha(a, p)\) for
\begin{itemize}
\item relative odd unitary Steinberg groups \(\stunit(R, \Delta; S, \Theta)\), where \(\delta \colon (S, \Theta) \to (R, \Delta)\) is a crossed module of odd form rings and \((R, \Delta)\) has a strong orthogonal hyperbolic family of rank at least \(3\) in the sense of \cite{CentralityBC} (this is a unitary analogue of generalized matrix rings and crossed modules of associative rings),
\item relative doubly laced Steinberg groups \(\stlin(\Phi; K, \mathfrak a)\) with \(\Phi\) of rank \(\geq 3\), where \(K\) is a unital commutative ring and \(\delta \colon \mathfrak a \to K\) is a crossed module of commutative rings.
\end{itemize}
The odd unitary case already gives a presentation of relative Steinberg groups associated with classical sufficiently isotropic reductive groups by \cite[theorem 4]{ClassicOFA}, so the second case is non-trivial only for the root system of type \(\mathsf F_4\). Some twisted forms of reductive groups of type $\mathsf A_\ell$ are constructed using non-strong orthogonal hyperbolic families on odd form rings (say, $\glin(\ell + 1, K)$ itself), but such groups may be constructed using generalized matrix rings, so they are considered in \cite{RelStLin}.

Actually, relative elementary subgroups of \(\stlin(\Phi; K)\) for doubly laced \(\Phi\) may be defined not only for ordinary ideals \(\mathfrak a \leqt K\), but for E. Abe's \cite{Abe} admissible pairs \((\mathfrak a, \mathfrak b)\), where \(\mathfrak a \leqt K\), \(2 \mathfrak a + \sum_{a \in \mathfrak a} Ka^2 \leq \mathfrak b \leq \mathfrak a\) is an intermediate group, and \(\mathfrak b k^2 \leq \mathfrak b\) for all \(k \in K\) in the case of type \(\mathsf C_\ell\) or \(\mathfrak b \leqt K\) in the cases of types \(\mathsf B_\ell\) and \(\mathsf F_4\). Such a pair naturally gives a subgroup \(\mathrm E(\Phi; \mathfrak a, \mathfrak b) \leq \mathrm G^{\mathrm{sc}}(\Phi, K)\) generated by \(x_\alpha(a)\) for short roots \(\alpha\) and \(a \in \mathfrak a\) and by \(x_\beta(b)\) for long roots \(\beta\) and \(b \in \mathfrak b\).

In order to study relative Steinberg groups associated with admissible pairs, we consider new families of Steinberg groups \(\stlin(\Phi; K, L)\), where \(\Phi\) is a doubly laced root system and \((K, L)\) is a \textit{pair of type} \(\mathsf B\), \(\mathsf C\), or \(\mathsf F\) respectively (the precise definition is given in section \ref{pairs-type}). Then admissible pairs are just crossed submodules of \((K, K)\) for a commutative unital ring \(K\). We also find relations between the generators \(z_\alpha(a, p)\) of
\begin{itemize}
\item relative doubly laced Steinberg groups \(\stlin(\Phi; K, L; \mathfrak a, \mathfrak b)\), where \(\Phi\) is a doubly laced root system of rank \(\geq 3\) and \((\mathfrak a, \mathfrak b) \to (K, L)\) is a crossed modules of pairs of type \(\mathsf B\), \(\mathsf C\), or \(\mathsf F\) respectively.
\end{itemize}

All relations between the generators \(z_\alpha(a, p)\) involve only the roots from root subsystems of rank \(2\), i.e. \(\mathsf A_2\), \(\mathsf{BC}_2\), \(\mathsf A_1 \times \mathsf A_1\), \(\mathsf A_1 \times \mathsf{BC}_1\) in the odd unitary case and \(\mathsf A_2\), \(\mathsf B_2\), \(\mathsf A_1 \times \mathsf A_1\) in the generalized Chevalley case.

Odd unitary groups were introduced by Victor Petrov  \cite{PetrovDef} in terms of odd form parameters on modules with hermitian forms. Later we generalized his definition \cite{ClassicOFA, CentralityBC} using so-called odd form algebras. These definitions, including crossed modules and Steinberg groups, are given in section 2. In section 3 we study in details parametrizations of root subgroups of an odd unitary group by the root system of type $\mathsf{BC}_\ell$, including changing it by a relative root system (<<root elimination>>). The family of generators and relations for relative unitary Steinberg groups is given in section 4 both in terms of abstract universal <<conjugacy calculus>> and in an explicit form, see proposition \ref{identities}. Unlike \cite{RelStLin}, we use the additional generators corresponding to relative roots from the beginning, so the relations take more compact form. Such generators may be expressed in terms of the ordinary generators $Z_{ij}$ and $Z_j$.

The proof that our relations are exhaustive occupies sections 5--7. In comparison with the linear case \cite{RelStLin} we prefer more abstract reasoning based on convex geometry instead of explicit calculations. The main idea is the same: in theorem \ref{root-elim} we show that our presentation is preserved under elimination of any root, then in theorem \ref{pres-stu} we construct actions of all root subgroups of the absolute Steinberg group on our abstract group, thus the abstract group may be identified with the relative Steinberg group. The main calculations essentially using the formalism of odd form rings (instead of abstract reasoning using only Chevalley's commutator formula) are given in lemmas \ref{ush-new-root} and \ref{sh-new-root}, where the missing root generators are constructed. If the rank of the orthogonal hyperbolic family (i.e. the corresponding root system) is at least $4$, then the proof of lemma \ref{sh-new-root} simplifies a bit at the very end. Namely, the technical lemma \ref{associator} turns out to be redundant.

For future possible applications we give in section \ref{even-case} an explicit presentation of relative unitary Steinberg groups in the case of matrix ``even'' unitary groups of Anthony Bak \cite{Bak}, including symplectic and even orthogonal groups (theorem \ref{pres-even-stu}). In the concluding sections 9--10 the main results are transferred from unitary groups to Chevalley groups. For the types $\mathsf B_\ell$ and $\mathsf C_\ell$ this reduces to a construction of corresponding odd form rings. In the exceptional case of $\mathsf F_4$ we use an argument in the spirit of \cite[theorem 9]{CentralityE}.

\section{Relative unitary Steinberg groups}

We use the group-theoretical notation \(\up gh = g h g^{-1}\) and \([g, h] = ghg^{-1}h^{-1}\). If a group \(G\) acts on a set \(X\), then we usually denote the action by \((g, x) \mapsto \up gx\). If \(X\) is itself a group, then \([g, x] = \up gx x^{-1}\) and \([x, g] = x (\up gx)^{-1}\) are the commutators in \(X \rtimes G\). A \textit{group-theoretical crossed module} is a homomorphism \(\delta \colon N \to G\) of groups such that there is a fixed action of \(G\) on \(N\), \(\delta\) is \(G\)-equivariant, and \(\up nn' = \up{\delta(n)}{n'}\) for \(n, n' \in N\).

The group operation in a \(2\)-step nilpotent group \(G\) is usually denoted by \(\dotplus\). If \(X_1\), \ldots, \(X_n\) are subsets of \(G\) containing \(\dot 0\) and \(\prod_i X_i \to G, (x_1, \ldots, x_n) \mapsto x_1 \dotplus \ldots \dotplus x_n\) is a bijection, then we write \(G = \bigoplus_i^\cdot X_i\).

Let \(A\) be an associative unital ring and \(\lambda \in A^*\). A map \(\inv{(-)} \colon A \to A\) is called a \(\lambda\)-\textit{involution} if it is an anti-automorphism, \(\inv{\inv x} = \lambda x \lambda^{-1}\), and \(\inv \lambda = \lambda^{-1}\). For a fixed \(\lambda\)-involution a map \(B \colon M \times M \to A\) for a module \(M_A\) is called a \textit{hermitian form} if it is biadditive, \(B(m, m'a) = B(m, m')a\), and \(B(m', m) = \inv{B(m, m')} \lambda\).

Now let \(A\) be an associative unital ring with a \(\lambda\)-involution and \(M_A\) be a module with a hermitian form \(B\). The \textit{Heisenberg group} of \(B\) is the set \(\Heis(B) = M \times A\) with the group operation \((m, x) \dotplus (m', x') = (m + m', x - B(m, m') + x')\). The multiplicative monoid \(A^\bullet\) acts on \(\Heis(B)\) from the right by \((m, x) \cdot y = (my, \inv y x y)\). An \(A^\bullet\)-invariant subgroup \(\mathcal L \leq \Heis(B)\) is called an \textit{odd form parameter} if
\[\{(0, x - \inv x \lambda) \mid x \in A\} \leq \mathcal L \leq \{(m, x) \mid x + B(m, m) + \inv x \lambda = 0\}.\]
The corresponding \textit{quadratic form} is the map \(q \colon M \to \Heis(B) / \mathcal L, m \mapsto (m, 0) \dotplus \mathcal L\). Finally, the \textit{unitary group} is
\[\unit(M, B, \mathcal L) = \{g \in \Aut(M_A) \mid B(gm, gm') = B(m, m'),\, q(gm) = q(m) \text{ for all } m, m' \in M\}.\]

Recall definitions from \cite{CentralityBC, ClassicOFA}. An \textit{odd form ring} is a pair \((R, \Delta)\), where \(R\) is an associative non-unital ring, \(\Delta\) is a group with the group operation \(\dotplus\), the multiplicative semigroup \(R^\bullet\) acts on \(\Delta\) from the right by \((u, a) \mapsto u \cdot a\), and there are maps \(\phi \colon R \to \Delta\), \(\pi \colon \Delta \to R\), \(\rho \colon \Delta \to R\) such that
\begin{itemize}
\item \(\phi\) is a group homomorphism, \(\phi(\inv aba) = \phi(b) \cdot a\);
\item \(\pi\) is a group homomorphism, \(\pi(u \cdot a) = \pi(u) a\);
\item \([u, v] = \phi(-\inv{\pi(u)} \pi(v))\);
\item \(\rho(u \dotplus v) = \rho(u) - \inv{\pi(u)} \pi(v) + \rho(v)\), \(\inv{\rho(u)} + \inv{\pi(u)} \pi(u) + \rho(u) = 0\), \(\rho(u \cdot a) = \inv a \rho(u) a\);
\item \(\pi(\phi(a)) = 0\), \(\rho(\phi(a)) = a - \inv a\);
\item \(\phi(a + \inv a) = \phi(\inv aa) = 0\) (in \cite{CentralityBC, ClassicOFA} we used the stronger axiom \(\phi(a) = \dot 0\) for all \(a = \inv a\));
\item \(u \cdot (a + b) = u \cdot a \dotplus \phi(\inv{\,b\,} \rho(u) a) \dotplus u \cdot b\).
\end{itemize}

Let \((R, \Delta)\) be an odd form ring. Its \textit{unitary group} is the set
\[\unit(R, \Delta) = \{g \in \Delta \mid \pi(g) = \inv{\rho(g)}, \pi(g) \inv{\pi(g)} = \inv{\pi(g)} \pi(g)\}\]
with the identity element \(1_{\unit} = \dot 0\), the group operation \(gh = g \cdot \pi(h) \dotplus h \dotplus g\), and the inverse \(g^{-1} = \dotminus g \cdot \inv{\pi(g)} \dotminus g\). The unitary groups of odd form rings in \cite{CentralityBC, ClassicOFA} are precisely the graphs of \(\pi \colon \unit(R, \Delta) \to R\) as subsets of \(R \times \Delta\).

An odd form ring \((R, \Delta)\) is called \textit{special} if the homomorphism \((\pi, \rho) \colon \Delta \to \Heis(R)\) is injective, where \(\Heis(R) = R \times R\) with the operation \((x, y) \dotplus (z, w) = (x + z, y - \inv xz + w)\). It is called \textit{unital} if \(R\) is unital and \(u \cdot 1 = u\) for \(u \in \Delta\). In other words, \((R, \Delta)\) is a special unital odd form ring if and only if \(R\) is a unital associative ring with \(1\)-involution and \(\Delta\) is an odd form parameter with respect to the \(R\)-module \(R\) and the hermitian form \(R \times R \to R, (x, y) \mapsto \inv xy\), where we identify \(\Delta\) with its image in \(\Heis(R)\).

If \(M_A\) is a module over a unital ring with a \(\lambda\)-involution, \(B\) is a hermitian form on \(M\), and \(\mathcal L\) is an odd form parameter, then there is a special unital odd form ring \((R, \Delta)\) such that \(\unit(M, B, \mathcal L) \cong \unit(R, \Delta)\), see \cite[section 2]{CentralityBC} or \cite[section 3]{ClassicOFA} for details. The even case is also considered in section \ref{even-case}.

We say that an odd form ring \((R, \Delta)\) \textit{acts} on an odd form ring \((S, \Theta)\) if there are multiplication operations \(R \times S \to S\), \(S \times R \to S\), \(\Theta \times R \to \Theta\), \(\Delta \times S \to \Delta\) such that \((S \rtimes R, \Theta \rtimes \Delta)\) is a well-defined odd form ring. There is an equivalent definition in terms of explicit axioms on the operations, see \cite[section 2]{CentralityBC}. For example, each odd form ring naturally acts on itself. Actions of \((R, \Delta)\) on \((S, \Theta)\) are in one-to-one correspondence with isomorphism classes of right split short exact sequences
\[(S, \Theta) \to (S \rtimes R, \Theta \rtimes \Delta) \leftrightarrows (R, \Delta),\]
since the category of odd form rings and their homomorphisms is algebraically coherent semi-abelian in the sense of \cite{AlgCoh}.

Let us call \(\delta \colon (S, \Theta) \to (R, \Delta)\) a \textit{precrossed module} of odd form rings if \((R, \Delta)\) acts on \((S, \Theta)\) and \(\delta\) is a homomorphism preserving the action of \((R, \Delta)\). Such objects are in one-to-one correspondence with \textit{reflexive graphs} in the category of odd form rings, i.e. tuples \(((R, \Delta), (T, \Xi), p_1, p_2, d)\), where \(p_1, p_2 \colon (T, \Xi) \to (R, \Delta)\) are homomorphisms with a common section \(d\). Namely, \((S, \Theta)\) corresponds to the kernel of \(p_2\) and \(\delta\) is induced by \(p_1\).

A precrossed module \(\delta \colon (S, \Theta) \to (R, \Delta)\) is a \textit{crossed module} of odd form rings if \textit{Peiffer identities}
\begin{itemize}
\item \(ab = \delta(a) b = a \delta(b)\) for \(a, b \in S\);
\item \(u \cdot a = \delta(u) \cdot a = u \cdot \delta(a)\) for \(u \in \Theta\), \(a \in S\)
\end{itemize}
hold. Equivalently, the corresponding reflexive graph is an \textit{internal category} (and even an \textit{internal groupoid}, necessarily in a unique way), i.e. there is a homomorphism
\[m \colon \lim\bigl( (T, \Xi) \xrightarrow{p_2} (R, \Delta) \xleftarrow{p_1} (T, \Xi) \bigr) \to (T, \Xi)\]
such that the homomorphisms from any \((I, \Gamma)\) to \((R, \Delta)\) and \((T, \Xi)\) form a set-theoretic category. See \cite{XMod} for details.

The unitary group \(\unit(R, \Delta)\) acts on \((R, \Delta)\) by automorphisms via
\[\up g a = \alpha(g)\, a\, \inv{\alpha(g)} \text{ for } a \in R, \enskip \up g u = (g \cdot \pi(u) \dotplus u) \cdot \inv{\alpha(g)} \text{ for } u \in \Delta,\]
where \(\alpha(g) = \pi(g) + 1 \in R \rtimes \mathbb Z\). The second formula also gives the conjugacy action of \(\unit(R, \Delta)\) on itself.

If \((R, \Delta)\) acts on \((S, \Theta)\), then \(\unit(R, \Delta)\) acts on \(\unit(S, \Theta)\) in the sense of groups and \(\unit(S \rtimes R, \Theta \rtimes \Delta) = \unit(S, \Theta) \rtimes \unit(R, \Delta)\). For any crossed module \(\delta \colon (S, \Theta) \to (R, \Delta)\) the induced homomorphism \(\delta \colon \unit(S, \Theta) \to \unit(R, \Delta)\) is a crossed module of groups.

Recall from \cite{CentralityBC} that a \textit{hyperbolic pair} in an odd form ring \((R, \Delta)\) is a tuple \(\eta = (e_-, e_+, q_-, q_+)\), where \(e_-\) and \(e_+\) are orthogonal idempotents in \(R\), \(\inv{e_-} = e_+\), \(q_\pm\) are elements of \(\Delta\), \(\pi(q_\pm) = e_\pm\), \(\rho(q_\pm) = 0\), and \(q_\pm = q_\pm \cdot e_\pm\). A sequence \(H = (\eta_1, \ldots, \eta_\ell)\) is called an \textit{strong orthogonal hyperbolic family} of rank \(\ell\), if \(\eta_i = (e_{-i}, e_i, q_{-i}, q_i)\) are hyperbolic pairs, the idempotents \(e_{|i|} = e_i + e_{-i}\) are orthogonal, and \(e_i \in R e_j R\) for all \(i, j \neq 0\).

From now on and until the end of section \ref{sec-pres}, we fix a crossed module \(\delta \colon (S, \Theta) \to (R, \Delta)\) of odd form rings and a strong orthogonal hyperbolic family \(H = (\eta_1, \ldots, \eta_\ell)\) in \((R, \Delta)\). We also use the notation
\[S_{ij} = e_i S e_j, \quad
\Theta^0_j = \{u \in \Theta \cdot e_j \mid e_k \pi(u) = 0 \text{ for all } 1 \leq |k| \leq \ell\}\]
for \(1 \leq |i|, |j| \leq \ell\), and similarly for the corresponding subgroups of \(R\) and \(\Delta\). Clearly,
\[S_{ij} R_{jk} = S_{ik}, \quad \Theta^0_i \cdot R_{ij} \dotplus \phi(S_{-j, j}) = \Theta^0_j.\]

An \textit{unrelativized Steinberg group} \(\stunit(S, \Theta)\) is the abstract group with the generators \(X_{ij}(a)\), \(X_j(u)\) for \(1 \leq |i|, |j| \leq \ell\), \(i \neq \pm j\), \(a \in S_{ij}\), \(u \in \Theta^0_j\), and the relations
\begin{align*}
X_{ij}(a) &= X_{-j, -i}(-\inv a); \\
X_{ij}(a)\, X_{ij}(b) &= X_{ij}(a + b); \\
X_j(u)\, X_j(v) &= X_j(u \dotplus v); \\
[X_{ij}(a), X_{kl}(b)] &= 1 \text{ for } j \neq k \neq -i \neq -l \neq j; \\
[X_{ij}(a), X_l(u)] &= 1 \text{ for } i \neq l \neq -j; \\
[X_{-i, j}(a), X_{ji}(b)] &= X_i(\phi(ab)); \\
[X_i(u), X_i(v)] &= X_i(\phi(-\inv{\pi(u)} \pi(v))); \\
[X_i(u), X_j(v)] &= X_{-i, j}(-\inv{\pi(u)} \pi(v)) \text{ for } i \neq \pm j; \\
[X_{ij}(a), X_{jk}(b)] &= X_{ik}(ab) \text{ for } i \neq \pm k; \\
[X_i(u), X_{ij}(a)] &= X_{-i, j}(\rho(u) a)\, X_j(\dotminus u \cdot (-a)).
\end{align*}

Of course, the group \(\stunit(S, \Theta)\) is functorial on \((S, \Theta)\). In particular, the homomorphism
\[\delta \colon \stunit(S, \Theta) \to \stunit(R, \Delta), X_{ij}(a) \mapsto X_{ij}(\delta(a)), X_j(u) \mapsto X_j(\delta(u))\]
is well-defined. There is also a canonical homomorphism
\begin{align*}
\stmap \colon \stunit(S, \Theta) &\to \unit(S, \Theta), \\
X_{ij}(a) &\mapsto T_{ij}(a) = q_i \cdot a \dotminus q_{-j} \cdot \inv a \dotminus \phi(a), \\
X_j(u) &\mapsto T_j(u) = u \dotminus \phi(\rho(u) + \pi(u)) \dotplus q_{-j} \cdot (\rho(u) - \inv{\pi(u)}).
\end{align*}

Let \(p_{i*} \colon \stunit(S \rtimes R, \Theta \rtimes \Delta) \to \stunit(R, \Delta)\) be the induced homomorphisms. The \textit{relative Steinberg group} is
\[\stunit(R, \Delta; S, \Theta) = \Ker(p_{2*}) / [\Ker(p_{1*}), \Ker(p_{2*})].\]
It is easy to see that it is a crossed module over \(\stunit(R, \Delta)\). The \textit{diagonal group} is
\[\diag(R, \Delta) = \{g \in \unit(R, \Delta) \mid g \cdot e_i \inv{\pi(g)} \dotplus q_i \cdot \inv{\pi(g)} \dotplus g \cdot e_i = \dot 0 \text{ for } 1 \leq |i| \leq \ell\},\]
it acts on \(\stunit(S, \Theta)\) by
\[\up g{T_{ij}(a)} = T_{ij}(\up ga), \quad \up g{T_j(u)} = T_j(\up gu).\]
Hence it also acts on the commutative diagram of groups
\[\xymatrix@R=30pt@C=90pt@!0{
\stunit(S, \Theta) \ar[r] \ar[dr]_{\delta} & \stunit(R, \Delta; S, \Theta) \ar[r]^{\stmap} \ar[d]_{\delta} & \unit(S, \Theta) \ar[r] \ar[d]_{\delta} & \mathrm{Aut}(S, \Theta) \\
& \stunit(R, \Delta) \ar[r]^{\stmap} & \unit(R, \Delta) \ar[ur] \ar[r] & \mathrm{Aut}(R, \Delta).
}\]

\section{Root systems of type \(\mathsf{BC}\)}

Let
\[\Phi = \{\pm \mathrm e_i \pm \mathrm e_j \mid 1 \leq i < j \leq \ell\} \cup \{\pm \mathrm e_i \mid 1 \leq i \leq \ell\} \cup \{\pm 2 \mathrm e_i \mid 1 \leq i \leq \ell\} \subseteq \mathbb R^\ell\]
be a \textit{root system} of type \(\mathsf{BC}_\ell\). For simplicity let also \(\mathrm e_{-i} = -\mathrm e_i\) for \(1 \leq i \leq \ell\). The \textit{roots} of \(\Phi\) are in one-to-one correspondence with the \textit{root subgroups} of \(\stunit(S, \Theta)\) as follows:
\begin{align*}
X_{\mathrm e_j - \mathrm e_i}(a) &= X_{ij}(a) \text{ for } a \in S_{ij}, i + j > 0,\\
X_{\mathrm e_i}(u) &= X_i(u) \text{ for } u \in \Theta^0_i,\\
X_{2 \mathrm e_i}(u) &= X_i(u) \text{ for } u \in \phi(S_{-i, i}).
\end{align*}
The image of \(X_\alpha\) is denoted by \(X_\alpha(S, \Theta)\). The \textit{Chevalley commutator formulas} from the definition of \(\stunit(S, \Theta)\) may be written as
\[[X_\alpha(\mu), X_\beta(\nu)] = \prod_{\substack{i \alpha + j \beta \in \Phi\\ i, j > 0}} X_{i\alpha + j\beta}(f_{\alpha \beta i j}(\mu, \nu))\]
for all non-anti-parallel \(\alpha, \beta \in \Phi\) and some universal expressions \(f_{\alpha \beta i j}\). It is also useful to set \(f_{\alpha \beta 0 1}(\mu, \nu) = \nu\) and \(f_{\alpha \beta 0 2}(\mu, \nu) = \dot 0\) if \(\beta\) is \textit{ultrashort} (i.e. of length \(1\)). In products we assume that the factor with such a root is the last one. The set of ultrashort roots is denoted by \(\Phi_{\mathrm{us}}\).

It is easy to see that \(\Ker(p_{2*}) \leq \stunit(S \rtimes R, \Theta \rtimes \Delta)\) is the group with the action of \(\stunit(R, \Delta)\) generated by \(\stunit(S, \Theta)\) with the additional relations
\[\up{X_\alpha(\mu)}{X_\beta(\nu)} = \prod_{\substack{i \alpha + j \beta \in \Phi\\ i \geq 0, j > 0}} X_{i\alpha + j\beta}(f_{\alpha \beta i j}(\mu, \nu))\]
for non-anti-parallel \(\alpha, \beta \in \Phi\), \(\mu \in R \cup \Delta\), \(\nu \in S \cup \Theta\). Hence the relative Steinberg group \(\stunit(R, \Delta; S, \Theta)\) is the crossed module over \(\stunit(R, \Delta)\) generated by \(\delta \colon \stunit(S, \Theta) \to \stunit(R, \Delta)\) with the same additional relations.

The \textit{Weyl group} \(\mathrm W(\mathsf{BC}_\ell) = (\mathbb Z / 2 \mathbb Z)^\ell \rtimes \mathrm S_\ell\) acts on the orthogonal hyperbolic family \(\eta_1, \ldots, \eta_\ell\) by permutations and sign changes (i.e. \((e_{-i}, e_i, q_{-i}, q_i) \mapsto (e_i, e_{-i}, q_i, q_{-i})\)), so the correspondence between roots and root subgroups is \(\mathrm W(\mathsf{BC}_\ell)\)-equivariant. Also, the hyperbolic pairs from \(H\) and the opposite ones are in one-to-one correspondence with the set of ultrashort roots, \(\eta_i\) corresponds to \(\mathrm e_i\), and \(\eta_{-i} = (e_i, e_{-i}, q_i, q_{-i})\) corresponds to \(\mathrm e_{-i}\) for \(1 \leq i \leq \ell\).

Recall that a subset \(\Sigma \subseteq \Phi\) is called \textit{closed} if \(\alpha, \beta \in \Sigma\) and \(\alpha + \beta \in \Phi\) imply \(\alpha + \beta \in \Sigma\). We say that a closed subset \(\Sigma \subseteq \Phi\) is \textit{saturated}, if \(\alpha \in \Sigma\) together with \(\frac 12 \alpha \in \Phi\) imply \(\frac 12 \alpha \in \Sigma\). If \(X \subseteq \Phi\), then \(\langle X \rangle\) is the smallest saturated subset of \(\Phi\) containing \(X\), \(\mathbb R X\) is the linear span of \(X\), and \(\mathbb R_+ X\) is the smallest convex cone containing \(X\). A saturated root subsystem \(\Psi \subseteq \Phi\) is a saturated subset such that \(\Psi = -\Psi\).

A closed subset \(\Sigma \subseteq \Phi\) is called \textit{special} if \(\Sigma \cap -\Sigma = \varnothing\). It is well-known that a closed subset of \(\Phi\) is special if and only if it is a subset of some system of positive roots. Hence the smallest saturated subset containing a special subset it also special. A root \(\alpha\) in a saturated special set \(\Sigma\) is called \textit{extreme} if it is indecomposable into a sum of two roots of \(\Sigma\) and in the case \(\alpha \in \Phi_{\mathrm{us}}\) the root \(2\alpha\) is not a sum of two distinct roots of \(\Sigma\). Every non-empty saturated special set contains an extreme root and if \(\alpha \in \Sigma\) is extreme, then \(\Sigma \setminus \langle \alpha \rangle\) is also a saturated special set. Notice that if \(\Sigma\) is a saturated special subset and \(u\) is an extreme ray of \(\mathbb R_+ \Sigma\) in the sense of convex geometry, then \(u\) contains an extreme root of \(\Sigma\).

If \(\Sigma \subseteq \Phi\) is a special subset, then the multiplication map
\[\prod_{\alpha \in \Sigma \setminus 2\Sigma} X_\alpha(S, \Theta) \to \stunit(S, \Theta)\]
is injective for any linear order on \(\Sigma \setminus 2\Sigma\) and its image \(\stunit(S, \Theta; \Sigma)\) is a subgroup of \(\stunit(S, \Theta)\) independent on the order. Moreover, the homomorphism \(\stunit(S, \Theta; \Sigma) \to \unit(S, \Theta)\) is injective. This follows from results of \cite[section 4]{CentralityBC}.

Let \(\Psi \subseteq \Phi\) be a saturated root subsystem. Consider the following binary relation on \(\Phi_{\mathrm{us}}\): \(\mathrm e_i \sim_\Psi \mathrm e_j\) if \(\mathrm e_i - \mathrm e_j \in \Psi \cup \{0\}\) and \(\mathrm e_j \notin \Psi\). Actually, this is a partial equivalence relation (i.e. symmetric and transitive), \(\mathrm e_i \sim_\Psi \mathrm e_j\) if and only if \(\mathrm e_{-i} \sim_\Psi \mathrm e_{-j}\), \(\mathrm e_i \not \sim_\Psi \mathrm e_{-i}\). Conversely, each partial equivalence relation on \(\Phi_{\mathrm{us}}\) with these properties arise from unique saturated root subsystem.

The image of \(\Phi \setminus \Psi\) in \(\mathbb R^\ell / \mathbb R \Psi\) is denoted by \(\Phi / \Psi\), in the case \(\Psi = \langle \alpha, -\alpha \rangle\) we write just \(\Phi / \alpha\). We associate with \(\Psi\) a new strong orthogonal hyperbolic family \(H / \Psi\) as follows. If \(E \subseteq \Phi_{\mathrm{us}}\) is an equivalence class with respect to \(\sim_\Psi\), then \(\eta_E\) is the sum of all hyperbolic pairs corresponding to the elements of \(E\), where a sum of two hyperbolic pairs is given by
\[(e_-, e_+, q_-, q_+) \oplus (e'_-, e'_+, q'_-, q'_+) = (e_- + e'_-, e_+ + e'_+, q_- + q'_-, q_+ + q'_+)\]
if \((e_- + e_+) (e'_- + e'_+) = 0\). The family \(H / \Psi\) consists of all \(\eta_E\), if we take only one equivalence class \(E\) from each pair of opposite equivalence classes (so \(H / \Psi\) does not contain opposite hyperbolic pairs). The Steinberg groups constructed by \(H / \Psi\) are denoted by \(\stunit(R, \Delta; \Phi / \Psi)\), \(\stunit(S, \Theta; \Phi / \Psi)\), and \(\stunit(R, \Delta; S, \Theta; \Phi / \Psi)\). In the case \(\Psi = \varnothing\) we obtain the original Steinberg groups. Now it is easy to see that \(\Phi / \Psi\) is a root system of type \(\mathsf{BC}_{\ell - \dim(\mathbb R \Psi)}\), it parametrizes the root subgroups of the corresponding Steinberg groups. Note that \(H / \Psi\) is defined only up to the action of \(\mathrm W(\Phi / \Psi)\).

Let us denote the map \(\Phi \setminus \Psi \to \Phi / \Psi\) by \(\pi_\Psi\). The preimage of a special subset of \(\Phi / \Psi\) is a special subset of \(\Phi\). There is a canonical group homomorphism \(F_\Psi \colon \stunit(S, \Theta; \Phi / \Psi) \to \stunit(S, \Theta; \Phi)\), it maps every root subgroup \(X_\alpha(S, \Theta)\) to \(\stunit(S, \Theta; \pi_\Psi^{-1}(\{\alpha, 2\alpha\} \cap \Phi / \Psi))\) in such a way that
\[\stmap \circ F_\Psi = \stmap \colon \stunit(S, \Theta; \Phi / \Psi) \to \unit(S, \Theta).\]
Of course, \(\{\alpha, 2\alpha\} \cap \Phi / \Psi\) coincides with \(\{\alpha, 2 \alpha\}\) for ultrashort \(\alpha\) and with \(\{\alpha\}\) otherwise. The map \(F_\Psi\) induces an isomorphism between \(\stunit(S, \Theta; \pi_\Psi^{-1}(\Sigma))\) and \(\stunit(S, \Theta; \Sigma)\) for any special $\Sigma \subseteq \Phi / \Psi$, so we identify such groups.

There are similarly defined natural homomorphisms \(F_\Psi \colon \stunit(R, \Delta; \Phi / \Psi) \to \stunit(R, \Delta; \Phi)\) and \(F_\Psi \colon \stunit(R, \Delta; S, \Theta; \Phi / \Psi) \to \stunit(R, \Delta; S, \Theta; \Phi)\). By \cite[propositions 1 and 2]{CentralityBC}, \(F_\Psi \colon \stunit(R, \Delta; \Phi / \alpha) \to \stunit(R, \Delta; \Phi)\) is an isomorphism for every root \(\alpha\) if \(\ell \geq 3\) (this also follows from theorem \ref{root-elim} proved below). The diagonal group \(\diag(R, \Delta; \Phi / \Psi)\) constructed by \(H / \Psi\) contains the root elements \(T_\alpha(\mu)\) for all \(\alpha \in \Psi\) and
\[F_\Psi\bigl(\up{T_\alpha(\mu)}{g}\bigr) = \up{X_\alpha(\mu)}{F_\Psi(g)} \in \stunit(R, \Delta; S, \Theta; \Phi)\]
for \(g \in \stunit(R, \Delta; S, \Theta; \Phi / \Psi)\), \(\alpha \in \Psi\).

Note that there is a one-to-one correspondence between the saturated root subsystems of \(\Phi\) containing a saturated root subsystem \(\Psi\) and the saturated root subsystems of \(\Phi / \Psi\). If \(\Psi \subseteq \Psi' \subseteq \Phi\) are two saturated root subsystems, then
\[\pi_{\Psi} \circ \pi_{\Psi' / \Psi} = \pi_{\Psi'} \colon \stunit(S, \Theta; \Phi / \Psi') \to \stunit(S, \Theta; \Phi).\]

Let \(e_{i \oplus j} = e_i + e_j\), \(q_{i \oplus j} = q_i \dotplus q_j\), \(e_{\ominus i} = e_{-i} + e_0 + e_i\). There are new root homomorphisms
\begin{align*}
X_{i, \pm (l \oplus m)} &\colon S_{i, \pm (l \oplus m)} = S_{i, \pm l} \oplus S_{i, \pm m} \\
&\to \stunit(S, \Theta; \Phi / (\mathrm e_m - \mathrm e_l)) \text{ for } i \notin \{0, \pm l, \pm m\}; \\
X_{\pm (l \oplus m), j} &\colon S_{\pm (l \oplus m), j} = S_{\pm l, j} \oplus S_{\pm m, j} \\
&\to \stunit(S, \Theta; \Phi / (\mathrm e_m - \mathrm e_l)) \text{ for } j \notin \{0, \pm l, \pm m\}; \\
X_{\pm(l \oplus m)} = X^0_{\pm(l \oplus m)} &\colon \Delta^0_{\pm(l \oplus m)} = \Theta^0_{\pm l} \dotoplus \phi(S_{\mp l, \pm m}) \dotoplus \Theta^0_{\pm m} \\
&\to \stunit(S, \Theta; \Phi / (\mathrm e_m - \mathrm e_l)); \\
X^{\ominus m}_j &\colon \Theta^{\ominus m}_j = q_{-m} \cdot S_{-m, j} \dotoplus \Theta^0_j \dotoplus q_m \cdot S_{mj} \\
&\to \stunit(S, \Theta; \Phi / \mathrm e_m) \text{ for } j \notin \{0, \pm m\}.
\end{align*}
The remaining root homomorphisms of \(\stunit(S, \Theta; \Phi / \mathrm e_m)\) and \(\stunit(S, \Theta; \Phi / (\mathrm e_m - \mathrm e_l))\) are denoted by the usual \(X_{ij}\) and \(X_j = X^0_j\).

\section{Conjugacy calculus}

Let us say that a group \(G\) has a \textit{conjugacy calculus} with respect to the strong orthogonal hyperbolic family \(H\) if there is a family of maps
\[\stunit(R, \Delta; \Sigma) \times \stunit(S, \Theta; \Phi) \to G, (g, h) \mapsto \up g{\{h\}_\Sigma}\]
parameterized by a saturated special subset \(\Sigma \subseteq \Phi\) such that
\begin{itemize}
\item[(Hom)] \(\up g{\{h_1 h_2\}_\Sigma} = \up g{\{h_1\}_\Sigma}\, \up g{\{h_2\}_\Sigma}\);
\item[(Sub)] \(\up g{\{h\}_{\Sigma'}} = \up g{\{h\}_\Sigma}\) if \(\Sigma' \subseteq \Sigma\);
\item[(Chev)] \(\up{g X_{\alpha}(\mu)}{\{X_{\beta}(\nu)\}_\Sigma} = \up g{\bigl\{\prod_{\substack{i \alpha + j \beta \in \Phi\\ i \geq 0, j > 0}} X_{i \alpha + j \beta}(f_{\alpha \beta i j}(\mu, \nu))\bigr\}_\Sigma}\) if \(\alpha, \beta\) are non-anti-parallel and \(\alpha \in \Sigma\);
\item[(XMod)] \(\up{X_\alpha(\delta(\mu))\, g}{\{h\}_\Sigma} = \up{\up 1{\{X_\alpha(\mu)\}_\Sigma}}{\bigl(\up g{\{h\}_\Sigma}\bigr)}\) if \(\alpha \in \Sigma\);
\item[(Conj)] \(\up{\up{g_1}{\{X_\alpha(\mu)\}_{\Sigma'}}}{\bigl(\up{F_\Psi(g_2)}{\{F_\Psi(h)\}_{\pi_\Psi^{-1}(\Sigma)}}\bigr)} = \up{F_\Psi(\up{\delta(f)}{g_2})}{\{F_\Psi(\up f {h})\}_{\pi_\Psi^{-1}(\Sigma)}}\) if \(\Psi \subseteq \Phi\) is a saturated root subsystem, \(\Sigma \subseteq \Phi / \Psi\) is a saturated special subset, \(\alpha \in \Psi\), \(\Sigma' \subseteq \Psi\), \(f = \up{\stmap(g_1)}{T_\alpha(\mu)} \in \unit(S, \Theta)\).
\end{itemize}
The axiom (Sub) implies that we may omit the subscript \(\Sigma\) in the maps \((g, h) \mapsto \up g{\{h\}_\Sigma}\).

If \(G\) has a conjugacy calculus with respect to \(H\), then we define the elements
\begin{align*}
Z_{ij}(a, p) &= \up{X_{ji}(p)}{\{X_{ij}(a)\}};
& X_{ij}(a) &= Z_{ij}(a, 0); \\
Z_j(u, s) &= \up{X_{-j}(s)}{\{X_j(u)\}};
& X_j(u) &= Z_j(u, \dot 0); \\
Z_{i, j \oplus k}(a, p) &= \up{X_{j \oplus k, i}(p)}{\{F_{\mathrm e_k - \mathrm e_j}(X_{i, j \oplus k}(a))\}};
& X_{i, j \oplus k}(a) &= Z_{i, j \oplus k}(a, 0); \\
Z_{i \oplus j, k}(a, p) &= \up{X_{k, i \oplus j}(p)}{\{F_{\mathrm e_j - \mathrm e_i}(X_{i \oplus j, k}(a))\}};
& X_{i \oplus j, k}(a) &= Z_{i \oplus j, k}(a, 0); \\
Z_{i \oplus j}(u, s) &= \up{X_{-(i \oplus j)}(s)}{\{F_{\mathrm e_j - \mathrm e_i}(X_{i \oplus j}(u))\}};
& X_{i \oplus j}(u) &= Z_{i \oplus j}(u, \dot 0);\\
Z^{\ominus i}_j(u, s) &= \up{X^{\ominus i}_{-j}(s)}{\{F_{\mathrm e_i}(X^{\ominus i}_j(u))\}};
& X^{\ominus i}_j(u) &= Z^{\ominus i}_j(u, \dot 0).
\end{align*}
of \(G\). Since \(\stunit(R, \Delta; S, \Theta)\) and \(\unit(S, \Theta)\) have natural conjugacy calculi with respect to \(H\), we use the notation \(Z_{ij}(a, p)\) and \(Z_j(u, s)\) for the corresponding elements in these groups.

\begin{prop} \label{identities}
Suppose that a group \(G\) has a conjugacy calculus with respect to \(H\). Then the following identities hold:
\begin{itemize}
\item[(Sym)] \(Z_{ij}(a, p) = Z_{-j, -i}(-\inv a, -\inv p)\), \(Z_{i \oplus j, k}(a, p) = Z_{-k, (-i) \oplus (-j)}(-\inv a, -\inv p) = Z_{j \oplus i, k}(a, p)\), \(Z_{i \oplus j}(u, s) = Z_{j \oplus i}(u, s)\), and \(Z_j^{\ominus i}(u, s) = Z_j^{\ominus (-i)}(u, s)\);
\item[(Add)] The maps \(Z_{ij}\), \(Z_j\), \(Z_{i \oplus j, k}\), \(Z_{i, j \oplus k}\), \(Z_{i \oplus j}\), \(Z^{\ominus i}_j\) are homomorphisms on the first variables.
\item[(Comm)]
 \begin{enumerate}
 \item \([Z_{ij}(a, p), Z_{kl}(b, q)] = 1\) if \(\pm i, \pm j, \pm k, \pm l\) are distinct,
 \item \([Z_{ij}(a, p), Z_l(b, q)] = 1\) if \(\pm i, \pm j, \pm l\) are distinct,
 \item \(\up{Z_{ij}(a, p)}{Z_{i \oplus j, k}(b, q)} = Z_{i \oplus j, k}\bigl(\up{Z_{ij}(a, p)} b, \up{Z_{ij}(a, p)} q\bigr)\),
 \item \(\up{Z_{ij}(a, p)}{Z_{i \oplus j}(u, s)} = Z_{i \oplus j}\bigl(\up{Z_{ij}(a, p)} u, \up{Z_{ij}(a, p)} s\bigr)\),
 \item \(\up{Z_i(u, s)}{Z^{\ominus i}_j(v, t)} = Z^{\ominus i}_j\bigl(\up{Z_i(u, s)} v, \up{Z_i(u, s)} t\bigr)\);
 \end{enumerate}
\item[(Simp)]
 \begin{enumerate}
 \item \(Z_{ik}(a, p) = Z_{i \oplus j, k}(a, p)\);
 \item \(Z_{ij}(a, p) = Z_{(-i) \oplus j}(\phi(a), \phi(p))\);
 \item \(Z_j(u, s) = Z^{\ominus i}_j(u, s)\);
 \end{enumerate}
\item[(HW)]
 \begin{enumerate}
 \item \(Z_{j \oplus k, i}\bigl(\up{T_{jk}(r)} a, p + q\bigr) = Z_{k, i \oplus j}\bigl(\up{T_{ij}(p)} a, \up{T_{ij}(p)}{(q + r)}\bigr)\) for \(a \in S_{ki}\), \(p \in R_{ij}\), \(q \in R_{ik}\), \(r \in R_{jk}\),
 \item \(Z_{-j \oplus -i}\bigl(\up{T_{ij}(q)}{(u \dotplus \phi(a))}, s \dotplus \phi(p) \dotplus t\bigr) = Z^{\ominus i}_{-j}\bigl(\up{T_i(s)}{(u \dotplus q_i \cdot a)}, \up{T_i(s)}{(q_{-i} \cdot p \dotplus t \dotplus q_i \cdot q)}\bigr)\) for \(u \in \Theta^0_{-j}\), \(a \in S_{i, -j}\), \(s \in \Delta^0_i\), \(p \in R_{-i, j}\), \(t \in \Delta^0_j\), \(q \in R_{ij}\);
 \end{enumerate}
\item[(Delta)]
 \begin{enumerate}
 \item \(Z_{ij}(a, \delta(b) + p) = \up{Z_{ji}(b, 0)}{Z_{ij}(a, p)}\),
 \item \(Z_j(u, \delta(v) \dotplus s) = \up{Z_{-j}(v, \dot 0)}{Z_j(u, s)}\).
 \end{enumerate}
\end{itemize}
Conversely, suppose that \(G\) is a group with the elements \(Z_{ij}(a, p)\), \(Z_j(u, s)\), \(Z_{i \oplus j, k}(a, p)\), \(Z_{i, j \oplus k}(a, p)\), \(Z_{i \oplus j}(u, s)\), \(Z^{\ominus i}_j(u, s)\) satisfying the identities above. Then \(G\) has a unique conjugacy calculus with respect to \(H\) such that the distinguished elements coincide with the corresponding expressions from the conjugacy calculus.
\end{prop}
\begin{proof}
If \(G\) has a conjugacy calculus, then all the identities may be proved by direct calculations. In particular, (Comm) follows from (Conj), (Delta) follows from (XMod), and the remaining ones follow from (Hom), (Sub), (Chev).

To prove the converse let us define the elements \(X_{i \oplus j, k}(a)\), \(X_{i, j \oplus k}(a)\), \(X_{i \oplus j}(u)\), \(X_j^{\ominus i}(u)\), \(X_{ij}(a)\), and \(X_j(u)\) as the corresponding generators with zero second argument. Notice that (Add) implies
\begin{align*}
Z_{ij}(0, p) &= 1, &
Z_{i \oplus j, k}(0, p) &= 1, &
Z_{i, j \oplus k}(0, p) &= 1, \\
Z_j(\dot 0, s) &= 1, &
Z_{i \oplus j}(\dot 0, s) &= 1, &
Z^{\ominus i}_j(\dot 0, s) &= 1.
\end{align*}
Together with (Sym), (Add), (Simp), and (HW) it follows that \(X_{i \oplus j, k}(a)\), \(X_{i, j \oplus k}(a)\), \(X_{i \oplus j}(u)\), and \(X^{\ominus i}_j(u)\) may be expressed in terms of the root elements \(X_{ij}(a)\) and \(X_j(u)\) in the natural way. Also, (Comm) implies that the root elements satisfy the Chevalley commutator formulas. It is also easy to see that (Sym), (Add), (Simp), and (HW) give some canonical expressions of all the distinguished elements in terms of \(Z_{ij}(a, p)\) and \(Z_j(u, s)\).

We explicitly construct the maps \((g, h) \mapsto \up g{\{h\}_\Sigma}\) by induction on \(|\Sigma|\), the case \(\Sigma = \varnothing\) is trivial. Simultaneously we show that \(\up g{\{h\}_\Sigma}\) evaluates to a distinguished element if \(\Sigma\) is strictly contained in a system of positive roots of a rank \(2\) saturated root subsystem and \(h \in \unit(S, \Theta; -\Sigma)\). Hence from now on we assume that \(\Sigma \subseteq \Phi\) is a saturated special subset and there are unique maps \(\up g{\{h\}_{\Sigma'}}\) for all \(\Sigma'\) with \(|\Sigma'| < |\Sigma|\) satisfying the axioms (Hom), (Sub), and (Chev).

Firstly, we construct \(\up g{\{X_\alpha(\mu)\}_\Sigma}\), where \(\alpha \in \Phi \setminus 2 \Phi\) if a fixed root. If there is an extreme root \(\beta \in \Sigma \setminus 2 \Sigma\) such that \(\beta \neq -\alpha\), then we define
\[\up{g X_\beta(\nu)}{\{X_\alpha(\mu)\}_\Sigma} = \up g{\bigl\{\up{X_\beta(\nu)}{X_\alpha(\mu)}\bigr\}_{\Sigma \setminus \langle \beta \rangle}}\]
for any \(g \in \stunit(R, \Delta; \Sigma \setminus \langle \beta \rangle)\), where \(\up{X_\beta(\nu)}{X_\alpha(\mu)}\) denotes the element \(\prod_{\substack{i\beta + j\alpha \in \Phi\\ i \geq 0, j > 0}} X_{i \beta + j \alpha}(f_{\beta \alpha i j}(\nu, \mu)) \in \stunit(S, \Theta)\). By (HW), this definition gives the distinguished element for appropriate \(\Sigma\).

Let us check that the definition is correct, i.e. if \(\beta, \gamma \in \Sigma \setminus 2 \Sigma\) are two extreme roots, \(\beta, \gamma, -\alpha\) are distinct, and \(g \in \stunit(R, \Delta; \Sigma \setminus (\langle \beta \rangle \cup \langle \gamma \rangle))\), then
\[\up{g X_\beta(\nu)}{\bigl\{\up{X_\gamma(\lambda)}{X_\alpha(\mu)}\bigr\}_{\Sigma \setminus \langle \gamma \rangle}} = \up{g [X_\beta(\nu), X_\gamma(\lambda)]\, X_\gamma(\lambda)}{\bigl\{\up{X_\beta(\nu)}{X_\alpha(\mu)}\bigr\}_{\Sigma \setminus \langle \beta \rangle}}.\]
If \(\langle \alpha, \beta, \gamma \rangle\) is special, then this claim easily follows. Else these roots lie in a common saturated root subsystem \(\Phi_0\) of type \(\mathsf A_2\) or \(\mathsf{BC}_2\). We may assume that \(\Sigma = \langle \beta, \gamma \rangle\), otherwise there is an extreme root in \(\Sigma\) but not in \(\Phi_0\). But then this is a simple corollary of (HW).

The above definition cannot be used if \(\Sigma = \langle -\alpha \rangle\). In this case we just define
\[\up{X_{ji}(p)}{\{X_{ij}(a)\}_{\langle \mathrm e_j - \mathrm e_i \rangle}} = Z_{ij}(a, p), \quad \up{X_{-i}(s)}{\{X_i(u)\}_{\langle \mathrm e_i \rangle}} = Z_i(u, s).\]

Now let us check that the map \((g, h) \to \up g{\{h\}_\Sigma}\) is well-defined, i.e. factors through the Steinberg relations on the second argument. By construction, it factors through the homomorphism property of the root elements. Let us check that if also factors through the Chevalley commutator formula for \([X_\alpha(\mu), X_\beta(\nu)]\), where \(\alpha, \beta\) are linearly independent roots. If there is an extreme root \(\gamma \in \Sigma\) such that \(\langle \alpha, \beta, \gamma \rangle\) is special, then we may apply the construction of \(\up g{\{h\}_\Sigma}\) via \(\gamma\). Otherwise let \(\Phi_0 \subseteq \Phi\) be the rank \(2\) saturated root subsystem containing \(\alpha\), \(\beta\), \(\Sigma\). If \(\Phi_0\) is of type \(\mathsf A_1 \times \mathsf A_1\) or \(\mathsf A_1 \times \mathsf{BC}_1\), then \(\Sigma \subseteq \langle -\alpha, -\beta \rangle\) and we may apply the corresponding case of (Comm). If \(\Phi_0\) is of type \(\mathsf A_2\) or \(\mathsf{BC}_2\) and \(\langle \alpha, \beta \rangle\) is not its subsystem of positive roots, then we just apply (Add).

Consider the case where \(\Phi_0\) is of type \(\mathsf A_2\), \(\alpha, \beta, \alpha + \beta \in \Phi_0\), and \(\Sigma \subseteq \langle -\alpha, -\beta \rangle\). Without loss of generality, \(\alpha = \mathrm e_j - \mathrm e_i\) and \(\beta = \mathrm e_k - \mathrm e_j\). Then
\begin{align*}
\up{X_{ji}(p)\, X_{ki}(q)\, X_{kj}(r)}{\{X_{ij}(a)\, X_{jk}(b)\}_{\Sigma}} &= X_{k, i \oplus j}\bigl(\up{T_{ji}(p)}{(qa)}\bigr)\, Z_{ij}(a, p)\, Z_{i \oplus j, k}\bigl(b, \up{T_{ji}(p)}{(q + r)}\bigr) \\
&= Z_{i \oplus j, k}\bigl(\up{T_{ji}(p)\, T_{ij}(a)}{b}, \up{T_{ji}(p)}{(q + r)}\bigr)\, X_{k, i \oplus j}\bigl(\up{T_{ji}(p)}{(qa)}\bigr)\, Z_{ij}(a, p) \\
&= \up{X_{ji}(p)\, X_{ki}(q)\, X_{kj}(r)}{\{X_{jk}(b)\, X_{ik}(ab)\, X_{ij}(a)\}_\Sigma}.
\end{align*}

The remaining case is where \(\Phi_0\) is of type \(\mathsf{BC}_2\), \(\alpha, \beta, \alpha + \beta, 2\alpha + \beta \in \Phi_0\), and \(\Sigma \subseteq \langle -\alpha, -\beta \rangle\). Without loss of generality, \(\alpha = \mathrm e_i\) and \(\beta = \mathrm e_j - \mathrm e_i\). We have
\begin{align*}
\up{X_{-i}(s)\, X_{i, -j}(p)\, X_{-j}(t)\, X_{ji}(q)}{\{X_i(u)\, X_{ij}(a)\}_\Sigma} &= Z_{i \oplus j}(u, s \dotplus \phi(p) \dotplus t)\, Z^{\ominus i}_j\bigl(\up{T_{-i}(s)}{(q_i \cdot a)}, \up{T_{-i}(s)}{(q_i \cdot p \dotplus t \dotminus q_{-i} \cdot \inv q)}\bigr)\\
&= X^{\ominus i}_{-j}\bigl(\up{T_{-i}(s)}{[t \dotplus q_i \cdot p, T_i(u)]}\bigr)\, Z_i(u, s)\\
&\cdot Z^{\ominus i}_j\bigl(\up{T_{-i}(s)}{(q_i \cdot a)}, \up{T_{-i}(s)}{(q_i \cdot p \dotplus t \dotminus q_{-i} \cdot \inv q)}\bigr) \\
&= Z^{\ominus i}_j\bigl(\up{T_{-i}(s)}{(q_i \cdot a \dotplus u \cdot a \dotplus q_{-i} \cdot \rho(u) a)}, \up{T_{-i}(s)}{(q_i \cdot p \dotplus t \dotminus q_{-i} \cdot \inv q)}\bigr) \\
&\cdot Z_{i \oplus j}(u, s \dotplus \phi(p) \dotplus t) \\
&= \up{X_{-i}(s)\, X_{i, -j}(p)\, X_{-j}(t)\, X_{ji}(q)}{\{X_{ij}(a)\, X_j(u \cdot a)\, X_{-i, j}(\rho(u) a)\, X_i(u)\}_\Sigma}.
\end{align*}

Clearly, our map \((g, h) \mapsto \up g{\{h\}_\Sigma}\) satisfy the required properties and it is unique. The axiom (XMod) follows from the Steinberg relations and (Delta) if \(\Sigma\) is one-dimensional, the general case follows from the construction of \(\up g{\{h\}_\Sigma}\).

To prove the axiom (Conj), without loss of generality \(\Sigma' \subseteq \langle -\alpha \rangle\) and \(\Psi = \mathbb R \alpha \cap \Phi\). Applying (Hom) and (Chev) multiple times to the term \(\up{F_\Psi(g_2)}{\{F_\Psi(h)\}_{\pi^{-1}_\Psi(\Sigma)}}\), we reduce to the case where \(\Sigma\) is also one-dimensional. This is precisely (Comm).
\end{proof}

From now on let \(\overline{\stunit}(R, \Delta; S, \Theta; \Phi)\) be the universal group with a conjugacy calculus with respect to \(H\). It is the abstract group with the presentation given by proposition \ref{identities}. Clearly, for a saturated root subsystem \(\Psi \subseteq \Phi\) there is a homomorphism \(F_\Psi \colon \overline{\stunit}(R, \Delta; S, \Theta; \Phi / \Psi) \to \overline{\stunit}(R, \Delta; S, \Theta; \Phi)\), i.e. every group with a conjugacy calculus with respect to \(H\) also has a canonical conjugacy calculus with respect to \(H / \Psi\). We have a sequence of groups
\[\stunit(S, \Theta; \Phi) \to \overline{\stunit}(R, \Delta; S, \Theta; \Phi) \to \stunit(R, \Delta; S, \Theta; \Phi)\]
with the action of \(\diag(R, \Delta; \Phi)\). Our goal for the next several sections is to prove that the right arrow is an isomorphism.

\section{Lemmas about odd form rings}

The difficult part in the proof that the right arrow is an isomorphism is to construct an action of \(\stunit(R, \Delta; \Phi)\) on \(\overline{\stunit}(R, \Delta; S, \Theta; \Phi)\). In order to do this, we prove that \(F_\alpha \colon \overline{\stunit}(R, \Delta; S, \Theta; \Phi / \alpha) \to \overline{\stunit}(R, \Delta; S, \Theta; \Phi)\) is an isomorphism for all roots \(\alpha\). Since the root elements \(T_\alpha(\mu) \in \diag(R, \Delta; \Phi / \alpha)\) acts on \(\overline{\stunit}(R, \Delta; S, \Theta; \Phi / \alpha)\) by automorphisms, they induce some automorphisms of \(\overline{\stunit}(R, \Delta; S, \Theta; \Phi)\). In this section we prove the surjectivity of \(F_\alpha\) and several preparatory results.

\begin{lemma} \label{ring-pres}
If \(i, j, k \neq 0\), then the multiplication map
\[S_{ik} \otimes_{e_k R e_k} R_{kj} \to S_{ij}\]
is an isomorphism.
\end{lemma}
\begin{proof}
Let \(e_j = \sum_m x_m y_m\) for some \(x_m \in R_{jk}\) and \(y_m \in R_{kj}\), they exists since \(e_j \in R e_k R\). Then a direct calculation show that
\[S_{ij} \to S_{ik} \otimes_{e_k R e_k} R_{kj}, a \mapsto \sum_m a x_m \otimes y_m\]
is the inverse to the map from the statement.
\end{proof}

\begin{lemma} \label{form-pres}
For any non-zero indices \(j\), \(k\) consider the group \(F\) with the generators \(u \boxtimes p\) for \(u \in \Theta^0_k\), \(p \in R_{kj}\), and \(\phi(a)\) for \(a \in S_{-j, j}\). The relations are
\begin{itemize}
\item \(\phi(a + b) = \phi(a) \dotplus \phi(b)\), \(\phi(a) = \phi(-\inv a)\);
\item \(u \boxtimes a \dotplus \phi(b) = \phi(b) \dotplus u \boxtimes a\), \([u \boxtimes a, v \boxtimes b]^\cdot = \phi(-\inv a \inv{\pi(u)} \pi(v) b)\);
\item \((u \dotplus v) \boxtimes a = u \boxtimes a \dotplus v \boxtimes a\), \(u \boxtimes (a + b) = u \boxtimes a \dotplus \phi(\inv{\,b\,} \rho(u) a) \dotplus u \boxtimes b\);
\item \(u \boxtimes ab = (u \cdot a) \boxtimes b\) for \(u \in \Theta^0_k\), \(a \in R_{kk}\), \(b \in R_{kj}\);
\item \(\phi(a) \boxtimes b = \phi(\inv{\,b\,} a b)\).
\end{itemize}
Then the homomorphism
\[f \colon F \to \Theta^0_j, u \boxtimes p \mapsto u \cdot p, \phi(a) \mapsto \phi(a)\]
is an isomorphism.
\end{lemma}
\begin{proof}
Let \(e_j = \sum_m x_m y_m\) for some \(x_m \in R_{jk}\) and \(y_m \in R_{kj}\). Consider the map
\[g \colon \Theta^0_j \to F, u \mapsto \sum_m^\cdot (u \cdot x_m) \boxtimes y_m \dotplus \phi\bigl(\sum_{m < m'} \inv{y_{m'}} \inv{x_{m'}} \rho(u) x_m y_m\bigr),\]
it is a section of \(f\). The relations of \(F\) easily imply that \(g\) is a homomorphism. Finally,
\begin{align*}
g(f(\phi(a))) &= \sum_m^\cdot (\phi(a) \cdot x_m) \boxtimes y_m \dotplus \phi\bigl(\sum_{m < m'} \inv{y_{m'}} \inv{x_{m'}} (a - \inv a) x_m y_m\bigr) = \phi(a); \\
g(f(u \boxtimes a)) &= \sum_m^\cdot (u \cdot a x_m) \boxtimes y_m \dotplus \phi\bigl(\sum_{m < m'} \inv{y_{m'}} \inv{x_{m'}} \inv a \rho(u) a x_m y_m\bigr) = u \boxtimes a. \qedhere
\end{align*}
\end{proof}

\begin{prop} \label{elim-sur}
The homomorphism
\[F_\alpha \colon \overline{\stunit}(R, \Delta; S, \Theta; \Phi / \alpha) \to \overline{\stunit}(R, \Delta; S, \Theta; \Phi)\]
is surjective if \(\alpha\) is a short root and \(\ell \geq 3\) or if \(\alpha\) is an ultrashort root and \(\ell \geq 2\). The homomorphism
\[F_{\Psi / \alpha} \colon \overline{\stunit}(R, \Delta; S, \Theta; \Phi / \Psi) \to \overline{\stunit}(R, \Delta; S, \Theta; \Phi / \alpha)\]
is also surjective if \(\alpha \in \Psi \subseteq \Phi\) is a root subsystem of type \(\mathsf A_2\) and \(\ell \geq 3\).
\end{prop}
\begin{proof}
By proposition \ref{identities}, \(\overline{\stunit}(R, \Delta; S, \Theta; \Phi)\) is generated by \(Z_{ij}(a, p)\) and \(Z_j(u, s)\). It suffices to show that they lie in the images of the homomorphisms. This is clear for the generators with the roots not in \(\{\alpha, -\alpha\}\) or \(\Psi\) respectively. For the remaining roots \(\beta\) it suffices to show that \(X_\beta(S, \Theta)\) lie in the images. This easily follows from lemmas \ref{ring-pres}, \ref{form-pres} and the identities
\begin{align*}
Z_{k, i \oplus j}\bigl(\up{T_{ji}(p)} a, \up{T_{ji}(p)} q\bigr) &= \up{X_{ji}(p)\, X_{ik}(q)}{\{X_{kj}(a)\}} = Z_{ij}(qa, p)\, X_{k, i \oplus j}\bigl(\up{T_{ji}(p)}a\bigr); \\
Z^{\ominus i}_j\bigl(\up{T_{-i}(s)}{(q_{-i} \cdot (-\inv a))}, \up{T_{-i}(s)}{(q_{-i} \cdot p)}\bigr) &= \up{X_{-i}(s)\, X_{-i, -j}(p)}{\{X_{-j, i}(a)\}} = Z_i(\varphi(pa), s)\, X^{\ominus i}_j\bigl(\up{T_{-i}(s)}{(q_{-i} \cdot (-\inv a))}\bigr); \\
Z^{\ominus i}_j\bigl(\up{T_{-i}(s)} u, \up{T_{-i}(s)}{(q_{-i} \cdot p)}\bigr) &= \up{X_{-i}(s)\, X_{-i, -j}(p)}{\{X_j(u)\}} = Z_i(u \cdot p, s)\, X_j\bigl(\up{T_{-i}(s)}{(u \dotminus q_{-i} \inv p \inv{\rho(u)})}\bigr). \qedhere
\end{align*}
\end{proof}

Of course, the proposition also implies that
\[F_\alpha \colon \overline{\stunit}(R, \Delta; S, \Theta; \Phi / \Psi) \to \overline{\stunit}(R, \Delta; S, \Theta; \Phi / \alpha)\]
is surjective if \(\ell \geq 3\) and \(\Psi\) is of type \(\mathsf{BC}_2\).

The final technical lemma is needed in the next section.

\begin{lemma}\label{associator}
Suppose that \(\ell = 3\). Let \(A\) be an abelian group and
\[\{-\}_{ij} \colon R_{1i} \otimes_{\mathbb Z} R_{ij} \otimes_{\mathbb Z} S_{j3} \to A\]
be homomorphisms for \(i, j \in \{-2, 2\}\). Suppose also that
\begin{itemize}
\item[(A1)] \(\{p \otimes qr \otimes a\}_{ik} = \{pq \otimes r \otimes a\}_{jk} + \{p \otimes q \otimes ra\}_{ij}\);
\item[(A2)] \(\{p \otimes q \otimes \inv pa\}_{-i, i} = 0\);
\item[(A3)] \(\{p \otimes qr \otimes a\}_{ij} = \{\inv q \otimes \inv pr \otimes a\}_{-i, j}\);
\item[(A4)] \(\{p \otimes q \otimes ra\}_{ij} = -\{\inv r \otimes \inv q \otimes \inv pa\}_{-j, -i}\).
\end{itemize}
Then \(\{x\}_{ij} = 0\) for all \(i\), \(j\), \(x\).
\end{lemma}
\begin{proof}
Let \(R_{|2|, |2|} = \sMat{R_{-2, -2}}{R_{-2, 2}}{R_{2, -2}}{R_{22}}\) and \(S_{|2|, 3} = \sCol{S_{-2, 3}}{S_{23}}\). For convenience we prove the claim for arbitrary left \(R_{|2|, |2|}\)-module \(S_{|2|, 3}\) instead of a part of a crossed module, where \(S_{\pm 1, 3} = R_{\pm 1, 2} \otimes_{R_{22}} S_{23}\) in (A2) and (A4). From the last two identities we get
\begin{align*}
\{px \otimes \inv yq \otimes a\}_{ij} &= \{py \otimes \inv xq \otimes a\}_{-i, j}; \tag{A5} \\
\{p \otimes qx \otimes \inv ya\}_{ji} &= \{p \otimes qy \otimes \inv xa\}_{j, -i} \tag{A6}
\end{align*}
for \(x \in R_{1i}\) and \(y \in R_{1, -i}\). This implies that
\[\{p \otimes qxyr \otimes a\}_{ij} = \{p \otimes qyxr \otimes a\}_{ij} \tag{A7}\]
for \(x, y \in R_{\pm 1, \pm 1}\). Let \((I, \Gamma) \leqt (R, \Delta)\) be the odd form ideal generated by \(xy - yx\) for \(x, y \in R_{11}\). From (A5)--(A7) we obtain that \(\{-\}_{ij}\) factor through \(R / I\) and \(S_{|2|, 3} / (I \cap R_{|2|, |2|}) S_{|2|, 3}\), so we may assume that \(R_{11}\) is commutative. Using (A4)--(A6) it is easy to see that
\[\{rx \otimes y e_1 z \otimes w e_1 a\}_{ij} = \{x \otimes yrz \otimes w e_1 a\}_{ij} = \{x \otimes y e_1 z \otimes wra\}_{ij} \tag{A8}\]
for \(r \in R_{11}\). By (A5) and (A6), it suffices to prove that \(\{x\}_{22} = 0\).

From (A2), (A4), (A5), and (A6) we get
\begin{align*}
\{px \otimes qx \otimes a\}_{22} &= 0; \tag{A9} \\
\{p \otimes yq \otimes ya\}_{22} &= 0 \tag{A10}
\end{align*}
for \(x \in R_{12}\) and \(y \in R_{21}\). Using (A1), (A8), (A9), (A10), and the linearizations of (A9) and (A10), we get
\begin{align*}
\{x \otimes y (pq)^3 z \otimes wa\}_{22} &= \{xyp \otimes q (pq)^2 z \otimes wa\}_{22} + \{(pq)^2 x \otimes yp \otimes qzwa\}_{22} \\
&= \{xyp \otimes qz' \otimes wpqa\}_{22} + \{x' \otimes ypqp \otimes qzwa\}_{22} \\
&= -\{xyz' \otimes qp \otimes w'a\}_{22} - \{x' \otimes qp \otimes y'zwa\}_{22} = 0
\end{align*}
for \(x, p, z \in R_{12}\), \(y, q, w \in R_{21}\), \(a \in S_{13}\), where \(x' = pqx - xqp\), \(y' = ypq - qpy\), \(z' = pqz - zqp\), \(w' = wpq - qpw\). The last equality follows from \(x'q = py' = z'q = pw' = 0\). It remains to notice that the elements \((pq)^3\) generate the unit ideal in \(R_{11}\).
\end{proof}

\section{Construction of root subgroups}

In this and the next sections \(\alpha = \mathrm e_m\) or \(\alpha = \mathrm e_m - \mathrm e_l\) for \(m \neq \pm l\) is a fixed root. We also assume that \(\ell \geq 3\). We are going to prove that
\[F_\alpha \colon \overline{\stunit}(R, \Delta; S, \Theta; \Phi / \alpha) \to \overline{\stunit}(R, \Delta; S, \Theta; \Phi)\]
is an isomorphism, i.e. that \(\overline{\stunit}(R, \Delta; S, \Theta; \Phi / \alpha)\) has a natural conjugacy calculus with respect to \(H\).

In this section we construct root elements \(\widetilde X_\beta(\mu) \in \overline{\stunit}(R, \Delta; S, \Theta; \Phi / \alpha)\) for \(\beta \in \Phi\) and prove the Steinberg relations for them. Let \(I = \{m, -m\}\) if \(\alpha = \mathrm e_m\) and \(I = \{m, -m, l, -l\}\) if \(\alpha = \mathrm e_m - \mathrm e_l\). If \(\beta \notin \mathbb R \alpha\), then there is a canonical choice for such elements, i.e.
\begin{align*}
\widetilde X_{ij}(a) &= X_{ij}(a) \text{ for } i, j \notin I; \\
\widetilde X_{ij}(a) &= \widetilde X_{-j, -i}(-\inv a) = X_{i, \pm (l \oplus m)}(a) \text{ for } \alpha = \mathrm e_m - \mathrm e_l, i \notin I, j \in \{\pm l, \pm m\}; \\
\widetilde X_{ij}(a) &= \widetilde X_{-j, -i}(-\inv a) = X_{\pm (l \oplus m)}(\phi(a)) \text{ for } \alpha = \mathrm e_m - \mathrm e_l, i = \mp l, j = \pm m; \\
\widetilde X_{ij}(a) &= \widetilde X_{-j, -i}(-\inv a) = X_j^{\ominus m}(q_i \cdot a) \text{ for } \alpha = \mathrm e_m, i = \pm m; \\
\widetilde X_j(u) &= X_j(u) \text{ for } \alpha = \mathrm e_m - \mathrm e_l, j \notin I; \\
\widetilde X_j(u) &= X_j^{\ominus m}(u) \text{ for } \alpha = \mathrm e_m, j \notin I; \\
\widetilde X_j(u) &= X_{\pm (l \oplus m)}(u) \text{ for } \alpha = \mathrm e_m - \mathrm e_l, j \in \{\pm l, \pm m\}.
\end{align*}
These elements satisfy all the Steinberg relations involving only roots from saturated special subsets \(\Sigma \subseteq \Phi\) disjoint with \(\mathbb R \alpha\). Similar elements also may be defined in \(\stunit(R, \Delta; \Phi / \alpha)\) and \(\stunit(S, \Theta; \Phi / \alpha)\). The conjugacy calculus with respect to \(H / \alpha\) gives a way to evaluate the elements \(\up{X_\beta(\mu)}{\{\widetilde X_\gamma(\nu)\}}\) in terms of \(\widetilde X_{i \beta + j \gamma}(\lambda)\) if \(\pm \alpha \notin \langle \beta, \gamma \rangle\) and \(\langle \beta, \gamma \rangle\) is special. Up to symmetry, it remains to construct the element \(\widetilde X_{lm}(a)\) for \(\alpha = \mathrm e_m - \mathrm e_l\) or \(\widetilde X_m(u)\) for \(\alpha = \mathrm e_m\) in \(\overline{\stunit}(R, \Delta; S, \Theta; \Phi / \alpha)\), as well as to prove the Steinberg relations involving \(\alpha\) and \(2\alpha\).

Consider the expressions \(\up{\widetilde X_\beta(\mu)}{\{\widetilde X_\gamma(\nu)\}}\), where \(\alpha\) is strictly inside the angle \(\mathbb R_+ \beta + \mathbb R_+ \gamma\). We expand them in terms of \(\widetilde X_{i \beta + j \gamma}(\lambda)\) adding new terms as follows. If \(\alpha = \mathrm e_m - \mathrm e_l\) let
\begin{align*}
\up{X_{li}(p)}{\{\widetilde X_{im}(a)\}} &= \widetilde X_{lm}^i(p, a)\, \widetilde X_{im}(a); \\
\up{X_{im}(p)}{\{\widetilde X_{li}(a)\}} &= \up i {\widetilde X}_{lm}(-a, p)\, \widetilde X_{li}(a); \\
\up{X_{-l}(s)}{\{\widetilde X_m(u)\}} &= \widetilde X_{lm}^\pi(s, \dotminus u)\, \widetilde X_m(u); \\
\up{X_m(s)}{\{\widetilde X_{-l}(u)\}} &= \up \pi{\widetilde X}_{lm}(u, s)\, \widetilde X_{-l}(u); \\
\up{X_{-l}(s)}{\{\widetilde X_{-l, m}(a)\}} &= \widetilde X^{-l}_{lm}(s, a)\, \widetilde X_m(\dotminus s \cdot (-a))\, \widetilde X_{-l, m}(a); \\
\up{X_m(s)}{\{\widetilde X_{l, -m}(a)\}} &= \up{-m}{\widetilde X}_{lm}(a, \dotminus s)\, \widetilde X_{-l}(\dotminus s \cdot \inv a)\, \widetilde X_{l, -m}(a); \\
\up{X_{l, -m}(p)}{\{\widetilde X_m(u)\}} &= \widetilde X_{-l}(u \cdot \inv p)\, \widetilde X^{-m}_{lm}(-p, \dotminus u)\, \widetilde X_m(u); \\
\up{X_{-l, m}(p)}{\{\widetilde X_{-l}(u)\}} &= \widetilde X_m(u \cdot (-p))\, \up{-l}{\widetilde X}_{lm}(u, -p)\, \widetilde X_{-l}(u); \\
\widetilde X_\alpha(S, \Theta) &= \bigl\langle \widetilde X^i_{lm}(p, a), \up i{\widetilde X}_{lm}(a, p), \widetilde X^\pi_{lm}(s, u), \up \pi {\widetilde X}_{lm}(u, s), \\
&\quad \widetilde X^{-l}_{lm}(s, a), \up{-m}{\widetilde X}_{lm}(a, s), \widetilde X^{-m}_{lm}(p, u), \up{-l}{\widetilde X}_{lm}(u, p) \bigr\rangle.
\end{align*}
In the case \(\alpha = \mathrm e_m\) let
\begin{align*}
\up{X_{-m, i}(p)}{\{\widetilde X_{im}(a)\}} &= \widetilde X_{-m, m}^i(p, a)\, \widetilde X_{im}(a); \\
\up{X_i(s)}{\{\widetilde X_{im}(a)\}} &= \widetilde X_{-i, m}(\rho(s) a)\, \widetilde X^i_m(\dotminus s, -a)\, \widetilde X_{im}(a); \\
\up{X_{im}(p)}{\{\widetilde X_i(u)\}} &= \up i {\widetilde X}_m(u, -p)\, \widetilde X_{-i, m}(-\rho(u) p)\, \widetilde X_i(u); \\
\widetilde X_\alpha(S, \Theta) &= \bigl\langle \widetilde X^i_{-m, m}(p, a), \widetilde X^i_m(s, a), \up i{\widetilde X}_m(u, p) \bigr\rangle.
\end{align*}

\begin{lemma} \label{elim-diag}
If \(g \in \widetilde X_\alpha(S, \Theta)\) and \(\beta \in \Phi / \alpha\), then
\[\up g{Z_\beta(\mu, \nu)} = Z_\beta\bigl(\up{\delta(\stmap(g))} \mu, \up{\delta(\stmap(g))} \nu\bigr).\]
\end{lemma}
\begin{proof}
Without loss of generality, \(g\) is a generator of \(\widetilde X_\alpha(S, \Theta)\). Let \(\Psi \subseteq \Phi\) be the saturated irreducible root subsystem of rank \(2\) involved in the definition of \(g\) (of type \(\mathsf A_2\) or \(\mathsf{BC}_2\)). If \(\beta \notin \Psi / \alpha\), then the claim follows from (Conj). Otherwise we apply proposition \ref{elim-sur} to
\[F_\beta \colon \overline{\stunit}(R, \Delta; S, \Theta; \Phi / \Psi) \to \overline{\stunit}(R, \Delta; S, \Theta; \Phi / \alpha)\]
in order to decompose \(Z_\beta(\mu, \nu)\) into the generators with roots not in \(\Psi / \alpha\).
\end{proof}

Let \(\mathrm{eval} \colon \widetilde X_\alpha(S, \Theta) \to S_{lm}\) or \(\mathrm{eval} \colon \widetilde X_\alpha(S, \Theta) \to \Theta^0_m\) be such that \(\stmap(g) = T_{lm}(\mathrm{eval}(g))\) or \(\stmap(g) = T_m(\mathrm{eval}(g))\) for \(g \in \widetilde X_\alpha(S, \Theta)\) depending on the choice of \(\alpha\). Lemma \ref{elim-diag} implies that
\[[\widetilde X_\beta(\mu), g] = \prod_{\substack{i \beta + j \alpha \in \Phi\\ i, j > 0}} X_{i \beta + j \alpha}(f_{\beta \alpha i j}(\mu, \mathrm{eval}(g)))\]
for all \(g \in \widetilde X_\alpha(S, \Theta)\) if \(\beta\) and \(\alpha\) are linearly independent.

Note that any generator of \(\widetilde X_\alpha(S, \Theta)\) is trivial if either of its arguments vanishes.

\begin{lemma} \label{ush-new-root}
Suppose that \(\alpha = \mathrm e_m\). Then there is unique homomorphism
\[\widetilde X_m \colon \Theta^0_m \to \overline{\stunit}(R, \Delta; S, \Theta; \Phi / \alpha)\]
such that \(g = \widetilde X_m(\mathrm{eval}(g))\) for all \(g \in \widetilde X_\alpha(S, \Theta)\).
\end{lemma}
\begin{proof}
Lemma \ref{elim-diag} implies that the elements \(\widetilde X^i_{-m, m}(p, a)\) lie in the center of \(\widetilde X_\alpha(S, \Theta)\). Expanding both sides of the identities
\begin{align*}
    \up{X_{-m, i}(p)}{\{\widetilde X_{im}(a + b)\}}
    &=
    \up{X_{-m, i}(p)}{\{\widetilde X_{im}(a)\}}\, \up{X_{-m, i}(p)}{\{\widetilde X_{im}(b)\}},
    \\
    \up{X_{-m, j}(p)\, X_{-m, i}(q)\, X_{ji}(r)}{\{\widetilde X_{im}(a)\}}
    &=
    \up{X_{-m, i}(q)\, \up{X_{-m, j}(p)}{X_{ji}(r)}\, X_{-m, j}(p)}{\{\widetilde X_{im}(a)\}},
    \\
    \up{X_{-m, i}(p)\, X_{jm}(-q)}{\{\widetilde X_{ij}(a)\}}
    &=
    \up{X_{jm}(-q)\, X_{-m, i}(p)}{\{\widetilde X_{ij}(a)\}}
\end{align*}
for \(i \neq \pm j\) using (Chev) (possibly with further applications of (Hom)) we obtain  It is easy to check that
\begin{align*}
    \widetilde X^i_{-m, m}(p, a + b)
    &=
    \widetilde X^i_{-m, m}(p, a)\, \widetilde X^i_{-m, m}(p, b),
    \\
    \widetilde X^i_{-m, m}(q + pr, a)
    &=
    \widetilde X^i_{-m, m}(q, a)\, \widetilde X^j_{-m, m}(p, ra),
    \\
    \widetilde X^i_{-m, m}(p, aq)
    &=
    \widetilde X^{-j}_{-m, m}(\inv q, -\inv a \inv p).
\end{align*}
From the second identity we easily get \(\widetilde X^i_{-m, m}(p + q, a) = \widetilde X^i_{-m, m}(p, a)\, \widetilde X^i_{-m, m}(q, a)\) and \(\widetilde X^i_{-m, m}(pq, a) = \widetilde X^j_{-m, m}(p, qa)\) for all \(i, j \neq \pm m\). Hence by lemma \ref{ring-pres} there is a unique homomorphism
\[\widetilde X_{-m, m} \colon S_{-m, m} \to \overline{\stunit}(R, \Delta; S, \Theta; \Phi / \alpha)\]
such that \(\widetilde X^k_{-m, m}(p, a) = \widetilde X_{-m, m}(pa)\) and \(\widetilde X_{-m, m}(a) = \widetilde X_{-m, m}(-\inv a)\).

Similarly expanding the identities
\begin{align*}
    \up{X_{im}(-p)}{\{\widetilde X_i(u \dotplus v)\}}
    &=
    \up{X_{im}(-p)}{\{\widetilde X_i(u)\}}\, \up{X_{im}(-p)}{\{\widetilde X_i(v)\}},
    \\
    \up{X_{jm}(-r)\, X_{-j, i}(p)}{\{\widetilde X_{ij}(a)\}}
    &=
    \up{\up{X_{jm}(-r)}{X_{-j, i}(p)}\, X_{jm}(-r)}{\{\widetilde X_{ij}(a)\}},
    \\
    \up{X_{im}(-p)\, X_{jm}(-q)\, X_{ji}(-r)}{\{\widetilde X_j(u)\}}
    &=
    \up{X_{jm}(-q)\, \up{X_{im}(-p)}{X_{ji}(-r)}\, X_{im}(-p)}{\{\widetilde X_j(u)\}}
\end{align*}
for \(i \neq \pm j\), we get
\begin{align*}
    \up i{\widetilde X}_m(u \dotplus v, p)
    &=
    \up i{\widetilde X}_m(u, p)\, \up i{\widetilde X}_m(v, p),
    \\
    \up j{\widetilde X_m}(\phi(pa), r)
    &=
    \widetilde X_{-m, m}(\inv r par),
    \\
    \up j{\widetilde X}_m(u, rp + q)
    &=
    \up i{\widetilde X}_m(u \cdot r, p)\, \widetilde X_{-m, m}(\inv q \rho(u) rp)\, \up j{\widetilde X}_m(u, q).
\end{align*}
The second identity may be generalized to \(\up i{\widetilde X}_m(\phi(a), p) = \widetilde X_{-m, m}(\inv pap)\). The last identity is equivalent to \(\up j{\widetilde X}_m(u, rp) = \up i{\widetilde X}_m(u \cdot r, p)\) and \(\up j{\widetilde X}_m(u, rp + q) = \up j{\widetilde X}_m(u, rp)\, \widetilde X_{-m, m}(\inv q \rho(u) rp)\, \up j{\widetilde X}_m(u, q)\) for \(i \neq \pm j\). Hence \(\up i{\widetilde X}_m(u, p + q) = \up i{\widetilde X}_m(u, p)\, \widetilde X_{-m, m}(\inv q \rho(u) p)\, \up i{\widetilde X}_m(u, q)\) and \(\up j{\widetilde X}_m(u, pq) = \up i{\widetilde X}_m(u \cdot p, q)\) for all \(i, j \neq \pm m\). Moreover, lemma \ref{elim-diag} implies that \([g, \up i{\widetilde X}_m(u, p)] = \widetilde X_{-m, m}\bigl(\inv p \inv{\pi(u)} \pi(\mathrm{eval}(g))\bigr)\) for all \(g \in \widetilde X_\alpha(S, \Theta)\). Now lemma \ref{form-pres} gives the unique homomorphism
\[\widetilde X_m \colon \Theta^0_m \to \overline{\stunit}(R, \Delta; S, \Theta; \Phi / \alpha)\]
such that \(\widetilde X_{-m, m}(a) = \widetilde X_m(\phi(a))\) and \(\up i{\widetilde X}_m(u, p) = \widetilde X_m(u \cdot p)\).

The remaining identity \(\widetilde X_m^i(s, a) = \widetilde X_m(s \cdot a)\) may be proved by expanding both sides of
\begin{align*}
    \up{X_i(\dotminus s)}{\{\widetilde X_{im}(-a - b)\}}
    &=
    \up{X_i(\dotminus s)}{\{\widetilde X_{im}(-a)\}}\,
    \up{X_i(\dotminus s)}{\{\widetilde X_{im}(-b)\}},
    \\
    \up{X_j(\dotminus s)\, X_{im}(-p)}{\{\widetilde X_{ji}(-a)\}}
    &=
    \up{X_{im}(-p)\, X_j(\dotminus s)}{\{\widetilde X_{ji}(-a)\}}
\end{align*}
for \(i \neq \pm j\) using (Chev). We get
\begin{align*}
    \widetilde X^i_m(s, a + b)
    &=
    \widetilde X^i_m(s, a)\, \widetilde X_{-m, m}(\inv{\,b\,} \rho(s) a)\, \widetilde X^i_m(s, b),
    \\
    \widetilde X^j_m(s, ap)
    &=
    \widetilde X_m(s \cdot ap)
    \text{ for } i \neq \pm j. \qedhere
\end{align*}
\end{proof}

\begin{lemma} \label{sh-new-root}
Suppose that \(\alpha = \mathrm e_m - \mathrm e_l\). Then there is unique homomorphism
\[\widetilde X_{lm} \colon S_{lm} \to \overline{\stunit}(R, \Delta; S, \Theta; \Phi / \alpha)\]
such that \(g = \widetilde X_{lm}(\mathrm{eval}(g))\) for all \(g \in \widetilde X_\alpha(S, \Theta)\).
\end{lemma}
\begin{proof}
First of all, there is a nontrivial element from the Weyl group \(\mathrm W(\Phi)\) stabilizing \(\alpha\), it exchanges \(\pm \mathrm e_l\) with \(\mp \mathrm e_m\). It gives a duality between the generators of \(\widetilde X_\alpha(S, \Theta)\) as follows:
\begin{align*}
\widetilde X^i_{lm}(p, a) &\leftrightarrow \up{-i}{\widetilde X_{lm}}(\inv a, -\inv p), &
\widetilde X^{-l}_{lm}(s, a) &\leftrightarrow \up{-m}{\widetilde X_{lm}}(-\inv a, \dotminus s), \\
\widetilde X^\pi_{lm}(s, u) &\leftrightarrow \up \pi{\widetilde X_{lm}}(\dotminus u, s), &
\widetilde X^{-m}_{lm}(p, u) &\leftrightarrow \up{-l}{\widetilde X_{lm}}(\dotminus u, -\inv p).
\end{align*}
So it suffices to prove only one half of the identities between the generators. Expanding both sides of
\begin{align*}
    \up{X_{i, -l}(q)\, X_{-l}(s)\, X_{li}(-p)}{\{\widetilde X_m(\dotminus u)\}}
    &=
    \up{X_{-l}(s)\, \up{X_{i, -l}(q)}{X_{li}(-p)}\, X_{i, -l}(q)}{\{\widetilde X_m(\dotminus u)\}},
    \\
    \up{X_{-m, i}(-p)\, X_{-l}(s)}{\{\widetilde X_{im}(a)\}}
    &=
    \up{X_{-l}(s)\, X_{-m, i}(-p)}{\{\widetilde X_{im}(a)\}}
\end{align*}
using (Chev) we obtain
\begin{align*}
    \widetilde X^\pi_{lm}(s \dotplus \phi(pq), u)
    &=
    \widetilde X^\pi_{lm}(s, u),
    \\
    \widetilde X^\pi_{lm}(s, \phi(pa))
    &=
    1;
\end{align*}
in particular, \(\widetilde X^\pi_{lm}(s, \phi(a)) = \widetilde X^\pi_{lm}(\phi(p), u) = 1\). From this and lemma \ref{elim-diag} we easily obtain that \(\widetilde X^i_{lm}(p, a)\), \(\widetilde X^\pi_{lm}(s, u)\), \(\widetilde X^{-l}_{lm}(s, a)\) lie in the center of \(\widetilde X_\alpha(S, \Theta)\).

Similarly expanding both sides of
\begin{align*}
    \up{X_{li}(p)\, X_{lj}(r)\, X_{ij}(q)}{\{\widetilde X_{jm}(a)\}}
    &=
    \up{X_{lj}(r)\, \up{X_{li}(p)}{X_{ij}(q)}\, X_{li}(p)}{\{\widetilde X_{jm}(a)\}},
    \\
    \up{X_{li}(p)}{\{\widetilde X_{im}(a + b)\}}
    &=
    \up{X_{li}(p)}{\{\widetilde X_{im}(a)\}}\,
    \up{X_{li}(p)}{\{\widetilde X_{im}(b)\}},
    \\
    \up{X_{-m, i}(-q)\, X_{li}(r)\, X_{l, -m}(p)}{\{\widetilde X_{im}(a)\}}
    &=
    \up{X_{li}(r)\, \up{X_{-m, i}(-q)}{X_{l, -m}(p)}\, X_{-m, i}(-q)}{\{\widetilde X_{im}(a)\}}
\end{align*}
for \(i \neq j\) we get
\begin{align*}
    \widetilde X_{lm}^i(p, qa)\,
    \widetilde X_{lm}^j(r, a)
    &=
    \widetilde X_{lm}^j(pq + r, a),
    \\
    \widetilde X^i_{lm}(p, a + b)
    &=
    \widetilde X^i_{lm}(p, a)\, \widetilde X^i_{lm}(p, b),
    \\
    \widetilde X^{-m}_{lm}(-p, \phi(qa))\,
    \widetilde X^i_{lm}(pq + r, a)
    &=
    \up{-i}{\widetilde X}_{lm}(p \inv a, \inv q)\,
    \widetilde X^i_{lm}(r, a).
\end{align*}
The first identity is equivalent to
\begin{align*}
\widetilde X_{lm}^i(p, qa)
&=
\widetilde X_{lm}^j(pq, a)
; \tag{B1} \\
\widetilde X_{lm}^j(pq + r, a)
&=
\widetilde X_{lm}^j(pq, a)\,
\widetilde X_{lm}^j(r, a)
\end{align*}
for \(i \neq \pm j\) and the third one is equivalent to
\begin{align*}
\widetilde X^{-m}_{lm}(-p, \phi(qa))\,
\widetilde X^i_{lm}(pq, a)
&=
\up{-i}{\widetilde X}_{lm}(p \inv a, \inv q); \tag{B2} \\
\widetilde X^i_{lm}(pq + r, a)
&=
\widetilde X^i_{lm}(pq, a)\,
\widetilde X^i_{lm}(r, a).
\end{align*}
It follows that the maps \(\widetilde X^i_{lm}(-, =)\) are biadditive.

Now we expand both sides of
\begin{align*}
    \up{X_{i, -l}(q)\, X_{li}(-p)}{\{\widetilde X_{-l, m}(a)\}}
    &=
    \up{\up{X_{i, -l}(q)}{X_{li}(-p)}\, X_{i, -l}(q)}{\{\widetilde X_{-l, m}(a)\}},
    \\
    \up{X_{-l}(s)\, X_i(t)}{\{\widetilde X_{im}(-a)\}}
    &=
    \up{\up{X_i(t)}{X_{-l}(s)}\, X_{-l}(s)}{\{\widetilde X_{im}(-a)\}}
\end{align*}
and obtain
\begin{align*}
    \widetilde X^{-l}_{lm}(\phi(pq), a)
    &=
    \widetilde X^i_{lm}(p, qa)\,
    \widetilde X^{-i}_{lm}(-\inv q, \inv pa), \tag{B3}
    \\
    \widetilde X^\pi_{lm}(s, t \cdot a)
    &=
    \widetilde X^i_{lm}(\inv{\pi(s)} \pi(t), a). \tag{B4}
\end{align*}
Using (Chev) and (B4) the identities
\begin{align*}
    \up{X_{li}(p)\, X_{-i}(s)}{\{\widetilde X_{-i, m}(a)\}}
    &=
    \up{\up{X_{li}(p)}{X_{-i}(s)}\, X_{li}(p)}{\{\widetilde X_{-i, m}(a)\}},
    \\
    \up{X_{-l, i}(p)\, X_{-l}(s)}{\{\widetilde X_{im}(a)\}}
    &=
    \up{\up{X_{-l, i}(p)}{X_{-l}(s)}\, X_{-l, i}(p)}{\{\widetilde X_{im}(a)\}},
    \\
    \up{X_{i, -l}(-p)\, X_i(t)}{\{\widetilde X_{-l, m}(a)\}}
    &=
    \up{\up{X_{i, -l}(-p)}{X_i(t)}\, X_{i, -l}(-p)}{\{\widetilde X_{-l, m}(a)\}}
\end{align*}
may be simplified to
\begin{align*}
    \widetilde X^i_{lm}(p, \rho(s) a)
    &=
    \widetilde X^{-i}_{lm}(p \rho(s), a), \tag{B5}
    \\
    \widetilde X^{-l}_{lm}(s, pa)
    &=
    \widetilde X^i_{lm}(\rho(s) p, a), \tag{B6}
    \\
    \widetilde X^{-l}_{lm}(t \cdot p, a)
    &=
    \widetilde X^i_{lm}(\inv p \rho(t), pa). \tag{B7}
\end{align*}

We are ready to construct a homomorphism \(\widetilde X_{lm}\) such that \(\widetilde X_{lm}^i(p, a) = \widetilde X_{lm}(pa)\) using lemma \ref{ring-pres}. If \(\ell \geq 4\), then it exists by (B1). Otherwise we may assume that \(l = 1\), \(m = 3\), and apply lemma \ref{associator} to
\[\{p \otimes q \otimes a\}_{ij} = \widetilde X_{lm}^j(pq, a)\, \widetilde X_{lm}^i(p, qa)^{-1}.\]
Namely, (B3) and (B7) imply (A2); (B3) and (B6) imply (A3); and (B3) implies (A4).

It remains to express the remaining generators via \(\widetilde X_{lm}\). For \(\widetilde X^\pi_{lm}\), \(\widetilde X^{-l}_{lm}\), and \(\widetilde X^{-m}_{lm}\) this follows from simplifying the identities
\begin{align*}
    \up{X_{-l}(s)}{\{\widetilde X_m(\dotminus u \dotminus v)\}}
    &=
    \up{X_{-l}(s)}{\{\widetilde X_m(\dotminus u)\}}\,
    \up{X_{-l}(s)}{\{\widetilde X_m(\dotminus v)\}},
    \\
    \up{X_{-l}(s)}{\{\widetilde X_{-l, m}(a + b)\}}
    &=
    \up{X_{-l}(s)}{\{\widetilde X_{-l, m}(a)\}}\,
    \up{X_{-l}(s)}{\{\widetilde X_{-l, m}(b)\}},
    \\
    \up{X_{li}(p)\, X_{l, -m}(-r)\, X_{i, -m}(-q)}{\{\widetilde X_m(\dotminus u)\}}
    &=
    \up{X_{l, -m}(-r)\, \up{X_{li}(p)}{X_{i, -m}(-q)}\, X_{li}(p)}{\{\widetilde X_m(\dotminus u)\}}
\end{align*}
using (Chev), (B4), (B6). We obtain
\begin{align*}
    \widetilde X^\pi_{lm}(s, v \dotplus u)
    &=
    \widetilde X^\pi_{lm}(s, u)\, \widetilde X^\pi_{lm}(s, v),
    \\
    \widetilde X^{-l}_{lm}(s, a + b)
    &=
    \widetilde X^{-l}_{lm}(s, a)\, \widetilde X^{-l}_{lm}(s, b),
    \\
    \widetilde X^{-m}_{lm}(pq + r, u)
    &=
    \widetilde X^i_{lm}(p, q \rho(u))\,
    \widetilde X^{-m}_{lm}(r, u).
\end{align*}
The dual generators may be expressed via \(\widetilde X_{lm}(\mu)\) using (B2).
\end{proof}

The elements \(\widetilde X_\alpha(\mu)\) constructed by lemmas \ref{ush-new-root} and \ref{sh-new-root} satisfy all the missing Steinberg relations in \(\overline \stunit(R, \Delta; S, \Theta; \Phi / \alpha)\). Also,
\[\widetilde X_\alpha(\mu)\, g\, \widetilde X_\alpha(\mu)^{-1} = \up{T_\alpha(\mu)}g \tag{*}\]
for any \(g\), this is easy to check for \(g = Z_\beta(\lambda, \nu)\) expressing \(\widetilde X_\alpha(\mu)\) via \(Z_\gamma(\lambda', \nu')\), where \(\gamma\) and \(\beta\) are linearly independent.

\section{Presentation of relative unitary Steinberg groups} \label{sec-pres}

We are ready to construct a conjugacy calculus with respect to \(H\) on \(\overline \stunit(R, \Delta; S, \Theta; \Phi / \alpha)\). For a special subset \(\Sigma \subseteq \Phi / \alpha\) and \(g \in \stunit(R, \Delta; \Sigma)\) let
\begin{align*}
\up g{\{\widetilde X_\beta(\mu)\}_{\pi_\alpha^{-1}(\Sigma)}} &= \up g{\{X_{\pi_\alpha(\beta)}(\mu')\}} \text{ for } \beta \notin \mathbb R \alpha; \\
\up g{\{\widetilde X_\beta(\mu)\}_{\pi_\alpha^{-1}(\Sigma)}} &= [g, T_\beta(\mu)]\, \widetilde X_\beta(\mu) \text{ for } \beta \in \mathbb R \alpha
\end{align*}
be the elements of \(\overline \stunit(R, \Delta; S, \Theta; \Phi / \alpha)\). Here \(\mu' \in S \cup \Theta\) is the element with the property \(T_{\pi_\alpha(\beta)}(\mu') = T_\beta(\mu) \in \unit(S, \Theta)\), \(T_\beta(\mu)\) naturally acts on \(\stunit(S \rtimes R, \Theta \rtimes \Delta; \Sigma)\) in the second formula.

\begin{lemma} \label{conj-comm}
Let \(\beta\) and \(\gamma\) be linearly independent roots of \(\Phi\) such that \(\alpha\) lies strictly inside the angle \(\mathbb R_+ \beta + \mathbb R_+ \gamma\). Then
\[\up{g X_\beta(\mu)}{\{\widetilde X_\gamma(\nu)\}_{\pi_\alpha^{-1}(\pi_\alpha(\beta))}} = \prod_{\substack{i \beta + j \gamma \in \Phi; \\ i \geq 0, j > 0}} \up g {\{\widetilde X_{i \beta + j \gamma}(f_{\beta \gamma ij}(\mu, \nu))\}_{\pi_\alpha^{-1}(\pi_\alpha(\beta))}}\]
for all \(g \in \stunit(R, \Delta; \pi_\alpha(\beta))\).
\end{lemma}
\begin{proof}
We evaluate the expressions
\begin{align*}
\up{X_{lj}(p)\, X_{l \oplus m, i}(q)\, X_{ji}(r)}{\{\widetilde X_{im}(a)\}} &\text{ for } \alpha = \mathrm e_m - \mathrm e_l; \\
\up{X_{i, -l}(p)\, X_{-(l \oplus m)}(s)\, X_{li}(q)\, X_{mi}(r)\, X_i(t)}{\{\widetilde X_m(u)\}} &\text{ for } \alpha = \mathrm e_m - \mathrm e_l; \\
\up{X_{-m, j}(p)\, X_j(s)\, X_{-i, j}(q)\, X_i^{\ominus m}(t)\, X_{ji}(r)}{\{\widetilde X_{im}(a)\}} &\text{ for } \alpha = \mathrm e_m
\end{align*}
in two ways as products of an element from \(\stunit(S, \Theta; \Sigma)\) for a special \(\Sigma \subseteq \Phi\) by an element of the type \(\up g{\{\widetilde X_\beta(\mu)\}}_{\pi_\alpha^{-1}(\pi_\alpha(\beta))}\) assuming that the indices and their opposites are distinct and non-zero. During the calculations it is convenient to put the factors \(\widetilde X_{im}(a)\) and \(\widetilde X_m(u)\) inside the curly brackets in the rightmost position. After cancellation we obtain particular cases of the required identity. This easily follows by considering the images in \(\unit(S, \Theta)\) without evaluating the first factors.

The remaining two cases follow from
\begin{align*}
\up{X_{-(l \oplus m)}(s)}{\{\widetilde X_{im}(pa)\}}\, &\up{X_{-(l \oplus m)}(s \dotminus t \cdot (-p) \dotplus \phi(qp))}{\{\widetilde X_{-l, m}(a)\}} \\
&\equiv
\up{X_{i, -l}(p)\, X_{-(l \oplus m)}(s \dotminus t \cdot (-p) \dotplus \phi(qp))\, X_{li}(q)\, X_i(t)}{\{\widetilde X_{-l, m}(a)\}} \\
&= \up{X_{li}(q + \inv p \inv{\rho(t)})\, X_i(t)\, X_{-(l \oplus m)}(s)\, \widetilde X_{i, -l}(p)}{\{\widetilde X_{-l, m}(a)\}} \\
&\equiv \up{X_{-(l \oplus m)}(s)}{\{\widetilde X_{im}(pa)\, \widetilde X_m(t \cdot pa \dotminus \phi(\inv a qpa))\, \widetilde X_{-l, m}(a)\}}
\end{align*}
up to a factor from \(\stunit(S, \Theta; \langle -\mathrm e_m, \mathrm e_m - \mathrm e_l, \mathrm e_i + \mathrm e_l \rangle)\) on the left for \(\alpha = \mathrm e_m - \mathrm e_l\) and
\begin{align*}
\up{X^{\ominus m}_i(s \dotplus q_{-m} \cdot qp)}{\{\widetilde X_{-i}(u)\}} &\equiv \up{\widetilde X_{ji}(p)\, X^{\ominus m}_i(s \dotplus q_{-m} \cdot qp)\, \widetilde X_{-m, j}(q)}{\{\widetilde X_{-i}(u)\}} \\
&= \up{\widetilde X_{-m, j}(q)\, X^{\ominus m}_i(s)\, \widetilde X_{ji}(p)}{\{\widetilde X_{-i}(u)\}} \\
&\equiv \up{X^{\ominus m}_i(s)}{\{\widetilde X_{im}(\rho(u) \inv p \inv q)\, \widetilde X_{-i}(u)\}}
\end{align*}
up to a factor from \(\stunit(S, \Theta; \langle \mathrm e_m, \mathrm e_i - \mathrm e_m, -\mathrm e_i - \mathrm e_j \rangle)\) on the left for \(\alpha = \mathrm e_m\).
\end{proof}

\begin{theorem} \label{root-elim}
Let \(\delta \colon (S, \Theta) \to (R, \Delta)\) be a crossed module of odd form rings, where \((R, \Delta)\) has a strong orthogonal hyperbolic family of rank \(\ell \geq 4\). Then
\[F_\alpha \colon \overline \stunit(R, \Delta; S, \Theta; \Phi / \alpha) \to \overline \stunit(R, \Delta; S, \Theta; \Phi)\]
is an isomorphism for any \(\alpha \in \Phi\).
\end{theorem}
\begin{proof}
We find the inverse homomorphism \(G_\alpha\) by providing a conjugacy calculus with respect to \(H\) on \(\overline \stunit(R, \Delta; S, \Theta; \Phi / \alpha)\). Our construction of the maps \((g, h) \mapsto \up g{\{h\}_\Sigma}\) depends on mutual alignment of \(\Sigma\) and \(\alpha\). We postpone checking (XMod) and (Conj) until the end of the proof since these properties are not used in the construction. If \(\Sigma = \varnothing\), then the required homomorphism \(\stunit(S, \Theta; \Phi) \to \overline \stunit(R, \Delta; S, \Theta; \Phi / \alpha)\) exists by lemmas \ref{ush-new-root} and \ref{sh-new-root}. Below we use the conjugacy calculus on \(\overline \stunit(R, \Delta; S, \Theta; \Phi / \alpha)\) and the known properties of \(\widetilde X_\alpha(\mu)\) without explicit references.

Suppose that \(\Sigma\) does not intersect \(\mathbb R \alpha\). We already have the map if \(h\) is a root element. In the subcase of two-dimensional \(\langle \alpha, \Sigma \rangle\) the maps \(\up g{\{-\}_\Sigma}\) are homomorphisms by lemma \ref{conj-comm} and
\[\up{\up g{\{\widetilde X_\alpha(\mu)\}_\Sigma}}{\bigl(\up g{\{X_\beta(\nu)\}_\Sigma}\bigr)} = \up{[g, X_\alpha(\delta(\mu))]\, \up{T_\alpha(\delta(\mu))}g}{\bigl\{X_\beta\bigl(\up{T_\alpha(\mu)} \nu\bigr)\bigr\}_\Sigma} = \up g{\bigl\{X_\beta\bigl(\up{T_\alpha(\mu)}\nu\bigr)\bigr\}_\Sigma}\]
for \(\beta \in \Phi / \alpha\) and \(\Sigma = \pi_\alpha^{-1} \langle -\beta \rangle\), this is a corollary of (*). They also satisfy (Chev) by lemma \ref{conj-comm}. If the dimension of \(\langle \alpha, \Sigma \rangle\) is larger than \(2\), then we prove (Hom) by induction on \(|\Sigma|\). Take non-anti-parallel roots $\beta, \gamma \in \Phi$, then either there is an extreme root $\delta \in \Sigma \setminus \langle -\beta, -\gamma \rangle$, or $\alpha \notin \mathbb R \Sigma + \mathbb R \beta + \mathbb R \gamma$ and (Hom) follows from the conjugacy calculus with respect to $H / \alpha$. We may assume that (Hom) holds for $\Sigma \setminus \langle \delta \rangle$, then
\begin{align*}
    \up{\up{g\, X_\delta(\mu)}{\{X_\beta(\nu)\}_\Sigma}}{\bigl(\up{g\, X_\delta(\mu)}{\{X_\gamma(\lambda)\}_\Sigma}\bigr)} &= \up{\up g{\{\up{X_\delta(\mu)}{X_\beta(\nu)}\}_{\Sigma \setminus \langle \delta \rangle}}}{\bigl(\up g{\{\up{X_\delta(\mu)}{X_\gamma(\lambda)}\}_{\Sigma \setminus \langle \delta \rangle}}\bigr)}\\
    &= \prod_{\substack{i \alpha + j \beta \in \Phi\\ i \geq 0, j > 0}} \up g{\{\up{X_\delta(\mu)}{X_{i \alpha + j \beta}(f_{\alpha \beta i j}(\nu, \lambda))}\}_{\Sigma \setminus \langle \delta \rangle}}\\
    &= \prod_{\substack{i \alpha + j \beta \in \Phi\\ i \geq 0, j > 0}} \up{g\, X_\delta(\mu)}{\{X_{i \alpha + j \beta}(f_{\alpha \beta i j}(\nu, \lambda))\}_\Sigma}
\end{align*}
for all $g \in \stunit(R, \Delta; \Sigma \setminus \langle \delta \rangle)$. Here a conjugation inside braces is an abbreviation of the Chevalley commutator formula.

From now on we may assume that \(\alpha \in \Sigma\) since the case \(-\alpha \in \Sigma\) is symmetric. Let us show that there is unique \(\up{(-)}{\{=\}_\Sigma}\) by induction on the smallest face \(\Gamma\) of the cone \(\mathbb R_+ \Sigma\) containing \(\alpha\). If \(\Gamma\) is one-dimensional, then \(\alpha\) is an extreme root of \(\Sigma\). We define \(\up{(-)}{\{=\}_\Sigma}\) as
\[\up{X_\alpha(\mu) g}{\{h\}_\Sigma} = \up{T_\alpha(\mu)}{\bigl(\up g{\{h\}_{\Sigma \setminus \langle \alpha \rangle}}\bigr)}\]
for \(g \in \stunit(R, \Delta; \Sigma \setminus \langle \alpha \rangle)\), it satisfies (Hom), (Sub), and (Chev).

In order to construct \(\up{(-)}{\{\widetilde X_\beta(\lambda)\}_\Sigma}\) for \(\dim \Gamma \geq 2\) take an extreme root \(\gamma \in \Gamma \cap \Sigma\) non-anti-parallel with \(\beta\) and let
\[\up{g\, X_\gamma(\mu)}{\{\widetilde X_\beta(\lambda)\}_\Sigma} = \prod_{\substack{i \gamma + j \beta \in \Phi\\ i \geq 0, j > 0}} \up g{\{\widetilde X_{i \gamma + j \beta}(f_{\gamma \beta i j}(\mu,
\lambda))\}_{\Sigma \setminus \langle \gamma \rangle}}\]
for \(g \in \stunit(R, \Delta; \Sigma \setminus \langle \gamma \rangle)\). Clearly, this definition is independent of \(\gamma\) unless \(\Gamma\) is two-dimensional and both \(\alpha\), \(-\beta\) lie in its relative interior. In this case let \(\gamma_1\), \(\gamma_2\) be the extreme roots of \(\Sigma \cap \Gamma\). Take a decomposition \(\widetilde X_\beta(\lambda) = \prod_t \up{X_\delta(\kappa_t)}{\{\widetilde X_{\varepsilon_t}(\lambda_t)\}}\), where \(\langle \delta, \beta, \varepsilon_t \rangle\) is special, two-dimensional, and not containing in \(\mathbb R \Gamma\); \(\delta\) and \(\varepsilon_t\) are on the opposite sides of \(\beta\); \(\langle \Sigma, \delta \rangle\) is special with a face \(\Gamma\). Again using conjugation inside braces as an abbreviation for the Chevalley commutator formula, for any \(g \in \stunit(R, \Delta; \Sigma \setminus (\langle \gamma_1 \rangle \cup \langle \gamma_2 \rangle))\) we have
\begin{align*}
&\up{g\, X_{\gamma_1}(\mu)}{\bigl\{\up{X_{\gamma_2}(\nu)}{\widetilde X_\beta(\lambda)}\bigr\}_{\Sigma \setminus \langle \gamma_2 \rangle}} = \prod_t \up{g\, X_{\gamma_1}(\mu)\, \up{X_{\gamma_2}(\nu)}{X_\delta(\kappa_t)}}{\bigl\{\up{X_{\gamma_2}(\nu)}{\widetilde X_{\varepsilon_t}(\lambda_t)}\bigr\}_{\langle \Sigma, \delta \rangle \setminus \langle \gamma_2 \rangle}} \\
&= \prod_t \up{g\, \up{X_{\gamma_1}(\mu)\, X_{\gamma_2}(\nu)}{X_\delta(\kappa_t)}}{\bigl\{\up{X_{\gamma_1}(\mu)\, X_{\gamma_2}(\nu)}{\widetilde X_{\varepsilon_t}(\lambda_t)}\bigr\}_{\langle \Sigma, \delta \rangle \setminus \langle \gamma_1, \gamma_2 \rangle}} \\
&= \prod_t \up{g\, [X_{\gamma_1}(\mu), X_{\gamma_2}(\nu)]\, X_{\gamma_2}(\nu)\, \up{X_{\gamma_1}(\mu)}{X_\delta(\kappa_t)}}{\bigl\{\up{X_{\gamma_1}(\mu)}{\widetilde X_{\varepsilon_t}(\lambda_t)}\bigr\}_{\langle \Sigma, \delta \rangle \setminus \langle \gamma_1 \rangle}} \\
&= \up{g\, [X_{\gamma_1}(\mu), X_{\gamma_2}(\nu)]\, X_{\gamma_2}(\nu)}{\bigl\{\up{X_{\gamma_1}(\mu)}{\widetilde X_\beta(\lambda)}\bigr\}_{\Sigma \setminus \langle \gamma_1 \rangle}},
\end{align*}
so the maps \(\up{(-)}{\{\widetilde X_\beta(\lambda)\}_\Sigma}\) are well-defined. By construction, they satisfy (Sub).

If \(\dim \Gamma \geq 3\), then it is easy to check that \(\up{(-)}{\{=\}_\Sigma}\) satisfy (Hom) and (Chev). In the case \(\dim \Gamma = 2\) the map \(\up{(-)}{\{=\}_\Sigma}\) also satisfy (Chev) and factors through the Chevalley commutator formula for \([\widetilde X_\beta(\lambda), \widetilde X_\gamma(\mu)]\) if at least one of \(\beta\), \(\gamma\) is not in \(\mathbb R \Gamma\). Otherwise let again \(\widetilde X_\beta(\lambda) = \prod_t \up{X_\delta(\kappa_t)}{\{\widetilde X_{\varepsilon_t}(\lambda_t)\}}\), so
\begin{align*}
\up{\up g{\{\widetilde X_\gamma(\mu)\}_\Sigma}}{\bigl(\up g{\{\widetilde X_\beta(\lambda)\}_\Sigma}\bigr)} &= \prod_t \up{\up g{\{\widetilde X_\gamma(\mu)\}_\Sigma}}{\bigl(\up{g\, X_\delta(\kappa_t)}{\{\widetilde X_{\varepsilon_t}(\lambda_t)\}_{\langle \Sigma, \delta \rangle}}\bigr)} \\
&= \prod_t \up{g\, X_\delta(\kappa_t)}{\bigl\{\up{X_\delta(\kappa_t)^{-1}\, X_\gamma(\delta(\mu))\, X_\delta(\kappa_t)}{\widetilde X_{\varepsilon_t}(\lambda_t)}\bigr\}_{\langle \Sigma, \delta \rangle}} \\
&= \prod_t \up g{\bigl\{\up{X_\gamma(\delta(\mu))\, X_\delta(\kappa_t)}{\widetilde X_{\varepsilon_t}(\lambda_t)}\bigr\}_{\langle \Sigma, \delta \rangle}} = \up g{\bigl\{\up{X_\gamma(\delta(\mu))}{\widetilde X_\beta(\lambda)}\bigr\}_\Sigma}.
\end{align*}

To sum up, we have the maps \(\up{(-)}{\{=\}_\Sigma}\). We check that they satisfy (XMod) in the form
\[\up{\widetilde X_\gamma(\mu)}{\bigl(\up g{\{\widetilde X_\beta(\nu)\}_\Sigma}\bigr)} = \up{X_\gamma(\delta(\mu))\, g}{\{\widetilde X_\beta(\nu)\}_\Sigma}\]
by induction on \(\Sigma\). If there is an extreme root \(\gamma \neq \delta \in \Sigma\) non-anti-parallel to \(\beta\), then we may apply (Chev) and the induction hypothesis. If \(\Sigma = \langle \gamma \rangle\), then the identity follows from the definition and (*). Finally, let \(\Sigma = \langle \gamma, -\beta \rangle\) be two-dimensional. Then
\[\up{\widetilde X_\gamma(\mu)}{\bigl(\up g {\{\widetilde X_\beta(\nu)\}_\Sigma}\bigr)} = \up{\widetilde X_\gamma(\mu)\, \up g{\{\widetilde X_\gamma(\mu)\}_\Sigma^{-1}}}{\bigl(\up g{\bigl\{\up{X_\gamma(\delta(\mu))}{\widetilde X_\beta(\nu)}\bigr\}_\Sigma}\bigr)} = \up{X_\gamma(\delta(\mu))\, g}{\{\widetilde X_\beta(\nu)\}_\Sigma}\]
for \(g \in \stunit(R, \Delta; \Sigma \setminus \langle \gamma \rangle)\) by (Chev) and the induction hypothesis.

It remains to check (Conj) in the form
\[\up{\widetilde Z_\beta(\mu, \nu)}{\bigl(\up{F_\beta(g)}{\{F_\beta(h)\}_{\pi_\beta^{-1}(\Sigma)}}\bigr)} = \up{F_\beta(\up fg)}{\bigl\{F_\beta\bigl(\up fh\bigr)\bigr\}_{\pi_\beta^{-1}(\Sigma)}}\]
for \(f = Z_\beta(\mu, \nu) \in \unit(S, \Theta)\). This is clear for \(\beta \in \mathbb R \alpha\), so we may assume that \(\alpha\) and \(\beta\) are linearly independent. Then
\[\up{\widetilde Z_\beta(\mu, \nu)}{\bigl(\up{F_{\langle \alpha, \beta \rangle}(g')}{\{F_{\langle \alpha, \beta \rangle}(h')\}}\bigr)} = \up{F_{\langle \alpha, \beta \rangle}(\up f{g'})}{\bigl\{F_{\langle \alpha, \beta \rangle}\bigl(\up f{h'}\bigr)\bigr\}}\]
by (Conj) from the conjugacy calculus with respect to \(H / \alpha\) and any \(\up{F_\beta(g)}{\{F_\beta(h)\}}\) may be expressed in the terms \(\up{F_{\langle \alpha, \beta \rangle}(g')}{\{F_{\langle \alpha, \beta \rangle}(h')\}}\) by proposition \ref{elim-sur}.

Now we have group homomorphisms \(F_\alpha\) and \(G_\alpha\). By proposition \ref{elim-sur}, the map \(F_\alpha\) is surjective, also \(G_\alpha \circ F_\alpha\) is the identity by construction. It follows that these maps are mutually inverse.
\end{proof}

\begin{theorem} \label{pres-stu}
Let \(\delta \colon (S, \Theta) \to (R, \Delta)\) be a crossed modules of odd form rings, where \((R, \Delta)\) has a strong orthogonal hyperbolic family of rank \(\ell \geq 3\). Then \(\overline \stunit(R, \Delta; S, \Theta) \to \stunit(R, \Delta; S, \Theta)\) is an isomorphism. In particular, the relative Steinberg group has the explicit presentation from proposition \ref{identities}.
\end{theorem}
\begin{proof}
Notice that
\[u \colon \overline \stunit(R, \Delta; S, \Theta) \to \stunit(R, \Delta; S, \Theta)\]
is surjective. Indeed, its image contains all generators and it is invariant under the actions of all \(X_\alpha(\mu) \in \stunit(R, \Delta)\) by proposition \ref{elim-sur}.

Let us construct an action of \(\stunit(R, \Delta)\) on \(G = \overline \stunit(R, \Delta; S, \Theta)\). For any \(\alpha \in \Phi\) an element \(X_\alpha(\mu)\) gives the canonical automorphism of \(\overline \stunit(R, \Delta; S, \Theta; \Phi / \alpha)\), so by theorem \ref{root-elim} it gives an automorphism of \(G\). We have to check that these automorphisms satisfy the Steinberg relations. Clearly, \(X_\alpha(\mu \dotplus \nu)\) gives the composition of the automorphisms associated with \(X_\alpha(\mu)\) and \(X_\alpha(\nu)\). If \(\alpha\) and \(\beta\) are linearly independent roots, then the automorphisms induced by the formal products \([X_\alpha(\mu), X_\beta(\nu)]\) and \(\prod_{\substack{i \alpha + j \beta \in \Phi \\ i, j > 0}} X_{i \alpha + j \beta}(f_{\alpha \beta i j}(\mu, \nu))\) coincide on the image of \(\overline \stunit(R, \Delta; S, \Theta; \Phi / \langle \alpha, \beta \rangle)\), so it remains to apply proposition \ref{elim-sur}.

Now it is easy to construct an \(\stunit(R, \Delta)\)-equivariant homomorphism
\[v \colon \stunit(R, \Delta; S, \Theta) \to \overline \stunit(R, \Delta; S, \Theta),\, X_\alpha(\mu) \mapsto X_\alpha(\mu).\]
We already know that \(u\) is surjective and clearly \(v \circ u\) is the identity, so \(u\) is an isomorphism.
\end{proof}

\section{The case of even unitary groups} \label{even-case}

Consider matrix ``even'' unitary groups $\unit(M, B, \mathcal L)$ \cite[\S1]{Bak}. They are defined in terms of a ring $A$ with a $\lambda$-involution and a so-called \textit{form parameter} $\Lambda \leq A$, i.e. a subgroup with the properties $\inv a \Lambda a \leq \Lambda$ for all $a \in A$ and
$$\{x - \inv x \lambda \mid x \in A\} \leq \Lambda \leq \{x \in A \mid x + \inv x \lambda = 0\}.$$ For simplicity we consider only the case of central $\lambda$, so $\inv{(-)}$ is an involution on $A$. Let $M = A^{2 \ell}$ with the basis $e_{-\ell}, \ldots, e_{-1}, e_1, \ldots, e_\ell$,
$$
B(m, m') = \sum_i \lambda^{\frac{1 - \varepsilon(i)}2} \inv{m_i} m'_{-i},\enskip
\mathcal L = \{(m, x) \mid x + \sum_{i = 1}^\ell \inv{m_i} m_{-i} \in \Lambda\},
$$
From now on $\varepsilon(i) = 1$ for $i > 0$ and $\varepsilon(i) = -1$ for $i < 0$. Thus the quadratic form $q(m) = \sum_{i = 1}^\ell \inv{m_i} m_{-i} + \Lambda$ takes values in $A / \Lambda \cong \Heis(B) / \mathcal L$.
The \textit{even unitary group} is $\unit(\ell, A, \Lambda) = \unit(M, B, \mathcal L)$, it is the subgroup of $\glin(2\ell, A)$ given by the equations
$$
g'_{ij} = \lambda^{\frac{\varepsilon(j) - \varepsilon(i)}2} \inv{g_{-j, -i}}, \qquad
\sum_{j > 0} \inv{g_{ji}} g_{-j, i} \in \Lambda,
$$
where $g'_{ij}$ are the entries of the inverse matrix to $g$.

The corresponding special unital odd form ring $(R, \Delta)$ (with the same unitary group) is as follows: $R = \mat(2\ell, A)$ with the involution
$\inv{(e_{ij} a)} = e_{-j, -i} \lambda^{\frac{\varepsilon(j) - \varepsilon(i)}2} \inv a$
and $\Delta$ consists of $(r, s) \in R \times R$ such that
\begin{align*}
    \sum_k \lambda^{\frac{1 - \varepsilon(k)}2} \inv{r_{ki}} r_{-k, j} + \lambda^{\frac{1 - \varepsilon(i)}2} s_{-i, j} + \lambda^{\frac{\varepsilon(j) - 1}2} \inv{s_{-j, i}} &= 0\\
    \sum_{j > 0} \inv{r_{ji}} r_{-j, i} + \lambda^{\frac{1 - \varepsilon(i)}2} s_{-i, i} &\in \Lambda.
\end{align*}
We index rows and columns in the matrix ring \(R\) by integers from \(-\ell\) to \(\ell\) excluding \(0\). Since $(R, \Delta)$ is special and unital, its unitary group may be identified with a subgroup of $\unit(R) = \{r \in R^* \mid r^{-1} = \inv{\,r\,}\}$ under the map $g \mapsto 1 + \pi(g)$. The isomorphism $\unit(R, \Delta) \cong \unit(2\ell, A, \Lambda)$ is given by the standard action of $\unit(R)$ on $M$.

There is a canonical strong orthogonal hyperbolic family in $(R, \Delta)$, namely, $e_i = e_{ii}$ for all $i$. Then $R_{ij} = e_{ij} A$, $\Delta^0_j = \{0\} \times e_{-j, j} \lambda^{\frac{\varepsilon(j) - 1}2} \Lambda$. Following \cite[\S 3.1]{BakVavilov}, we use $A$ and $\lambda^{\frac{\varepsilon(j) - 1}2} \Lambda$ as domains for elementary transvections, i.e.
$$
T_{ij}(x) = 1 + e_{ij} x - e_{-j, -i} \lambda^{\frac{\varepsilon(j) - \varepsilon(i)}2} \inv x \text{ for } i \neq \pm j, \enskip
T_j(u) = 1 + e_{-j, j} u.
$$
In the even unitary case the root subgroups corresponding to $\mathrm e_i$ and $2 \mathrm e_i$ coincide, so we do not need the ultrashort roots. Thus a bit different notation \(T_{-j, j}(u) = T_j(u)\) is used in the paper \cite{BakVavilov}, also $\inv u = -\lambda^{-\varepsilon(j)} u$. The Steinberg relations are given in \cite[\S 3.2]{BakVavilov}, the non-trivial ones are
\begin{align*}
    T_{ij}(a) &= T_{-j, -i}(-\lambda^{\frac{\varepsilon(j) - \varepsilon(i)}2} \inv a),\\
    [T_{-i, j}(a), T_{ji}(b)] &= T_i(ab - \lambda^{\varepsilon(i)} \inv{ (ab) }),\\
    [T_{ij}(a), T_{jk}(b)] &= T_{ik}(ab),\\
    [T_i(u), T_{ij}(a)] &= T_{-i, j}(ua)\, T_j(-\lambda^{\frac{\varepsilon(j) - \varepsilon(i)}2} \inv a ua).
\end{align*}
Here all indices are non-zero and with distinct absolute values.

Recall that a \textit{form ideal} \cite[\S 4.1]{BakVavilov} $(I, \Gamma) \leqt (A, \Lambda)$ consists of an ideal $I \leqt A$ closed under the \(\lambda\)-involution and a subgroup $\Gamma \leq I$ such that $\inv a \Gamma a \leq \Gamma$ for all $a \in A$ and
$$\{x - \inv x \lambda \mid x \in I\} + \sum_{x \in I} \inv x \Lambda x \leq \Gamma \leq I \cap \Lambda.$$
The corresponding odd form ideal $(J, \Omega) \leqt (R, \Delta)$ consists of $J = \mat(2\ell, I)$ and the set $\Omega$ of the pairs $(r, s) \in \Delta \cap (J \times J)$ such that
$$
\sum_{j > 0} \inv{r_{ji}} r_{-j, i} + \lambda^{\frac{1 - \varepsilon(i)}2} s_{-i, i} \in \Gamma
$$
Actually, it is easy to see that all odd form ideals of $(R, \Delta)$ are of this form.

The relative even unitary Steinberg group is denoted by $\stunit(2\ell; A, \Lambda; I, \Gamma)$. Generalized root subgroups corresponding to the relative roots are parameterized by the following groups:
\begin{align*}
    V_{i, j \oplus k}(I, \Gamma) &= I \oplus I, &
    V_{i \oplus j}(I, \Gamma) &= \lambda^{\frac{\varepsilon(i) - 1}2} \Gamma \oplus I \oplus \lambda^{\frac{\varepsilon(j) - 1}2} \Gamma,\\
    V_{i \oplus j, k}(I, \Gamma) &= I \oplus I, &
    V_j^{\ominus i}(I, \Gamma) &= I \dotoplus \lambda^{\frac{\varepsilon(j) - 1}2} \Lambda \dotoplus I.
\end{align*}
In the last case the group operation is induced from the embedding into $\Delta$, i.e.
\begin{align*}
    (a \dotoplus u \dotoplus b) &\dotplus (c \dotoplus v \dotoplus d)\\
    &= (a + c) \dotoplus (u + \lambda^{\frac{\varepsilon(j) + \varepsilon(i)} 2} \inv{\,c\,} b - \lambda^{\frac{\varepsilon(j) - \varepsilon(i)} 2} \inv{\,b\,} c + v) \dotoplus (b + d).
\end{align*}
The corresponding elements of the type $Z$ are
\begin{align*}
    Z_{i, j \oplus k}(a \oplus b, p \oplus q) &= \up{X_{ji}(p)\, X_{ki}(q)}{(X_{ij}(a)\, X_{ik}(b))},\\
    Z_{i \oplus j, k}(a \oplus b, p \oplus q) &= \up{X_{ki}(p)\, X_{kj}(q)}{(X_{ik}(a)\, X_{jk}(b))},\\
    Z_{i \oplus j}(u \oplus a \oplus v, s \oplus p \oplus t) &= \up{X_{-i}(s)\, X_{i, -j}(p)\, X_{-j}(t)}{(X_i(u)\, X_{-i, j}(a)\, X_j(v))},\\
    Z_j^{\ominus i}(a \dotoplus u \dotoplus b, p \dotoplus s \dotoplus q) &= \up{X_{i, -j}(p)\, X_{-j}(s)\, X_{-i, -j}(q)}{(X_{-i, j}(a)\, X_j(u)\, X_{ij}(b))}.
\end{align*}

If $i$ and $j$ are indices such that $i \neq \pm j$, then there is an embedding $\glin(2, A) \to \unit(2\ell, A, \Lambda)$ given by
\begin{align*}
    g e_i &= e_i g_{11} + e_j g_{21}, &
    g e_{-i} &= e_{-i} \inv{g'_{11}} + e_{-j} \lambda^{\frac{\varepsilon(j) - \varepsilon(i)}2} \inv{g'_{12}},\\
    g e_j &= e_i g_{12} + e_j g_{22}, &
    g e_{-j} &= e_{-i} \lambda^{\frac{\varepsilon(i) - \varepsilon(j)}2} \inv{g'_{21}} + e_{-j} \inv{g'_{22}},\\
    g e_k &= e_k \text{ for } k \notin \{i, -i, j, -j\}
\end{align*}
in the standard presentation, $t_{12}(x) = 1 + e_{12} x \mapsto T_{ij}(x)$, $t_{21}(y) = 1 + e_{21} y \mapsto T_{ji}(y)$. The group $\glin(2, A)$ acts on $V_{i, j \oplus k}(I, \Gamma)$ by the embedding into the unitary group via the indices $(j, k)$,
$$
\up g{(a \oplus b)} = (a g'_{11} + b g'_{21}) \oplus (a g'_{12} + b g'_{22}).
$$
Similarly, $\glin(2, A)$ acts on $V_{i \oplus j, k}(I, \Gamma)$ and $V_{i \oplus j}(I, \Gamma)$ by embeddings to $\unit(2\ell, A, \Lambda)$ via $(i, j)$, i.e.
\begin{align*}
    \up g{(a \oplus b)} &= (g_{11} a + g_{12} b) \oplus (g_{21} a + g_{22} b),\\
    \up g{(u \oplus a \oplus v)} &= (\inv{g'_{11}} u g'_{11} + \inv{g'_{11}} a g'_{21} - \lambda^{\varepsilon(i)} \inv{g'_{21}} \inv a g'_{11} + \lambda^{\frac{\varepsilon(i) - \varepsilon(j)}2} \inv{g'_{21}} v g'_{21})\\
    &\oplus (\inv{g'_{11}} u g'_{12} + \inv{g'_{11}} a g'_{22} - \lambda^{\varepsilon(i)}  \inv{g'_{21}} \inv a g'_{12} + \lambda^{\frac{\varepsilon(i) - \varepsilon(j)}2} \inv{g'_{21}} v g'_{22})\\
    &\oplus (\lambda^{\frac{\varepsilon(j) - \varepsilon(i)}2} \inv{g'_{12}} (u g'_{12} + a g'_{22}) - \lambda^{\frac{\varepsilon(i) + \varepsilon(j)}2} \inv{g'_{22}} \inv a g'_{12} + \inv{g'_{22}} v g'_{22}).
\end{align*}
Finally, for every $i$ the group $\unit(2, A, \Lambda)$ embeds to $\unit(2\ell, A, \Lambda)$ via
\begin{align*}
    g e_{-i} &= e_{-i} g_{-1, -1} + e_i \lambda^{\frac{1 - \varepsilon(i)}2} g_{1, -1},\\
    g e_i &= e_{-i} \lambda^{\frac{\varepsilon(i) - 1}2} g_{-1, 1} + e_i g_{11},\\
    g e_j &= e_j \text{ for } j \neq \pm i
\end{align*}
in the standard presentation, $T_1(u) \mapsto T_i( \lambda^{\frac{\varepsilon(i) - 1}2} u )$, $T_{-1}(v) \mapsto T_{-i}( \lambda^{\frac{1 - \varepsilon(i)}2} v)$. The group $\unit(2, A, \Lambda)$ acts on $V_j^{\ominus i}(I, \Gamma)$ via this embedding, namely,
\begin{align*}
    \up g{(a \dotoplus u \dotoplus b)} &= (g_{-1, -1} a + \lambda^{\frac{\varepsilon(i) - 1}2} g_{-1, 1} b)\\ &\dotoplus (u + \lambda^{\frac{\varepsilon(j) - 1}2} (\inv a v a - c + \lambda \inv{\,c\,} + \inv{\,b\,} w b))\\
    &\dotoplus (\lambda^{\frac{1 - \varepsilon(i)}2} g_{1, -1} a + g_{11} b),
\end{align*}
where $v = \lambda \inv{g_{-1, -1}} g_{1, -1} \in \Lambda$, $w = \lambda \inv{g_{-1, 1}} g_{11} \in \Lambda$, and $c = \lambda^{\frac{\varepsilon(i) - 1}2} \inv a \inv{g_{1, -1}} g_{-1, 1} b$.

\begin{theorem} \label{pres-even-stu}
    Let $A$ be a ring with a $\lambda$-involution, where $\lambda$ is central, and a form parameter $\Lambda$, $(I, \Gamma) \leqt (A, \Lambda)$ be a form ideal. For any $\ell \geq 3$ the group $\stunit(2\ell; A, \Lambda; I, \Gamma)$ is generated by $Z_{ij}(a, p)$, $Z_j(u, s)$, $Z_{i, j \oplus k}(A, P)$, $Z_{i \oplus j, k}(A, P)$, $Z_{i \oplus j}(U, S)$, $Z^{\ominus i}_j(A, P)$, where all indices have distinct absolute values and the arguments are taken from $I$, $\lambda^{\frac{\varepsilon(j) - 1}2} \Gamma$, and the groups of the type \(V\) defined above. The only relations between these generators are the following:
    \begin{itemize}
        \item[(Sym)] \begin{enumerate}
            \item $Z_{ij}(a, p) = Z_{-j, -i}(-\mu \inv a, -\mu^{-1} \inv p)$, where $\mu = \lambda^{\frac{\varepsilon(j) - \varepsilon(i)}2}$,
            \item $Z_{i \oplus j, k}(a \oplus b, p \oplus q) = Z_{-k, (-i) \oplus (-j)}(-(\mu a \oplus \nu b), -(\mu^{-1} p \oplus \nu^{-1} q)) = Z_{j \oplus i, k}(b \oplus a, q \oplus p)$, where $\mu = \lambda^{\frac{\varepsilon(k) - \varepsilon(i)}2}$ and $\nu = \lambda^{\frac{\varepsilon(k) - \varepsilon(j)}2}$,
            \item $Z_{i \oplus j}(u \oplus a \oplus v, s \oplus p \oplus t) = Z_{j \oplus i}(v \oplus (-\mu \inv a) \oplus u, t \oplus (-\mu^{-1} \inv p) \oplus s)$, where $\mu = \lambda^{\frac{\varepsilon(i) + \varepsilon(j)}2}$,
            \item $Z_j^{\ominus i}(a \dotoplus u \dotoplus b, p \dotoplus s \dotoplus q) = Z_j^{\ominus(-i)}(b \dotoplus (u - x + \lambda^{\varepsilon(j)} \inv x) \dotoplus a, q \dotoplus (s - y + \lambda^{-\varepsilon(j)} \inv y) \dotoplus p)$, where $x = \lambda^{\frac{\varepsilon(i) + \varepsilon(j)}2} \inv a b$ and $y = \lambda^{\frac{-\varepsilon(i) - \varepsilon(j)}2} \inv p q$;
        \end{enumerate}
        \item[(Add)] all generators are homomorphisms on the first argument;
        \item[(Comm)] \begin{enumerate}
            \item $[Z_{ij}(a, p), Z_{kl}(b, q)] = 1$ for distinct $\pm i$, $\pm j$, $\pm k$, $\pm l$,
            \item $[Z_{ij}(a, p), Z_l(u, s)] = 1$ for distinct $\pm i$, $\pm j$, $\pm l$,
            \item $\up{Z_{ij}(a, p)}{Z_{i \oplus j, k}(B, Q)} = Z_{i \oplus j, k}(\up{Z_{12}(a, p)} B, \up{Z_{12}(a, p)} Q)$,
            \item $\up{Z_{ij}(a, p)}{Z_{i \oplus j}(U, S)} = Z_{i \oplus j}(\up{Z_{12}(a, p)}U, \up{Z_{21}(-\mu \inv a, -\mu^{-1} \inv p)} S)$, where $\mu = \lambda^{\frac{\varepsilon(j) - \varepsilon(i)}2}$,
            \item $\up{Z_i(u, s)}{Z^{\ominus i}_j(V, T)} = Z^{\ominus i}_j(\up{Z_1(\mu u, \mu^{-1} s)}V, \up{Z_{-1}(\nu u, \nu^{-1} s)}T)$, where $\mu = \lambda^{\frac{1 - \varepsilon(i)}2}$ and $\nu = \lambda^{\frac{1 + \varepsilon(i)}2}$;
        \end{enumerate}
        \item[(Simp)] \begin{enumerate}
            \item $Z_{ik}(a, p) = Z_{i \oplus j, k}(a \oplus 0, p \oplus 0)$,
            \item $Z_{ij}(a, p) = Z_{(-i) \oplus j}(0 \oplus a \oplus 0, 0 \oplus (- \lambda^{\frac{\varepsilon(i) - \varepsilon(j)}2} p) \oplus 0)$,
            \item $Z_j(u, s) = Z^{\ominus i}_j(0 \dotoplus u \dotoplus 0, 0 \dotoplus s \dotoplus 0)$;
        \end{enumerate}
        \item[(HW)] \begin{enumerate}
            \item $Z_{j \oplus k, i}(\up{T_{12}(r)}{(0 \oplus a)}, p \oplus q) = Z_{k, i \oplus j}(\up{T_{12}(p)}{(a \oplus 0)}, \up{T_{12}(p)}{(q \oplus r)})$,
            \item $Z_{j \oplus i}(\up{T_{12}(- \lambda^{\frac{\varepsilon(i) - \varepsilon(j)}2} \inv q)}{(u \oplus (-\lambda^{\frac{\varepsilon(i) + \varepsilon(j)}2} \inv a) \oplus 0)}, t \oplus (-\lambda^{\frac{-\varepsilon(i) - \varepsilon(j)}2} \inv p) \oplus s)\\ = Z^{\ominus i}_j(\up{T_{-1}(\lambda^{\frac{\varepsilon(i) - 1}2} s)}{(a \dotoplus u \dotoplus 0)}, \up{T_1(\lambda^{\frac{1 + \varepsilon(i)}2} s)}{(p \dotoplus t \dotoplus q)})$;
        \end{enumerate}
        \item[(Delta)] $\up{Z_{ji}(b, 0)}{Z_{ij}(a, p)} = Z_{ij}(a, p + b)$, $\up{Z_{-j}(v, 0)}{Z_j(u, s)} = Z_j(u, s + v)$.
    \end{itemize}
\end{theorem}
\begin{proof}
    This is a direct corollary of theorem \ref{pres-stu}.
\end{proof}

In the case of symplectic groups the ring $A$ is commutative, $\lambda = -1$, $\inv a = a$, $\Lambda = A$. Also, $\lambda^{\frac{\varepsilon(i) \pm 1}2} = \mp \varepsilon(i)$, $\lambda^{\frac{\varepsilon(i) \pm \varepsilon(j)}2} = \mp \varepsilon(i) \varepsilon(j)$. The actions of $\mathrm{SL}(2, A) = \mathrm{Sp}(2, A)$ on $V_{i \oplus j}(I, \Gamma)$ and $V^{\ominus i}_j(I, \Gamma)$ may be written in the simpler forms
\begin{align*}
    \up g{(u \oplus a \oplus v)} &= (g_{22}^2 u - 2 g_{21} g_{22} a + \varepsilon(i) \varepsilon(j) g_{21}^2 v)\\
    &\oplus (-g_{12} g_{22} u + (g_{11} g_{22} + g_{12} g_{21}) a - \varepsilon(i) \varepsilon(j) g_{11} g_{21} v)\\
    &\oplus (\varepsilon(i) \varepsilon(j) g_{12}^2 u - 2 \varepsilon(i) \varepsilon(j) g_{11} g_{12} a + g_{11}^2 v).\\
    \up g{(a \dotoplus u \dotoplus b)} &= (g_{11} a + \varepsilon(i) g_{12} b)\\ &\dotoplus (u - \varepsilon(j) (g_{11} g_{21} a^2 + 2 \varepsilon(i) g_{12} g_{21} a b + g_{12} g_{22} b^2))\\
    &\dotoplus (\varepsilon(i) g_{21} a + g_{22} b).
\end{align*}

In the case of even orthogonal groups (considered also in \cite[theorem 2]{RelStLin}) $A$ is commutative, $\lambda = 1$, $\inv a = a$, $\Lambda = 0$. Thus $V_{i \oplus j}(I) = I$, $V^{\ominus i}_j(I) = I \oplus I$ and the root elements with the long roots are trivial. The action of $\mathrm{SL}(2, A)$ on $V_{i \oplus j}(I)$ is also trivial. The generators $Z_i(u, s)$, $Z_{i \oplus j}(a, p)$ (coinciding with $Z_{-i, j}(a, -p)$), and $Z_j^{\ominus i}(a \oplus b, p \oplus q)$ (coinciding with $Z_{-i, j}(a, -p)\, Z_{ij}(b, -q)$) are redundant. The relations between the remaining generators take the form
\begin{itemize}
    \item[(Sym)] \begin{enumerate}
        \item $Z_{ij}(a, p) = Z_{-j, -i}(-a, -p)$,
        \item $Z_{i \oplus j, k}(a \oplus b, p \oplus q) = Z_{-k, (-i) \oplus (-j)}(-(a \oplus b), -(p \oplus q)) = Z_{j \oplus i, k}(b \oplus a, q \oplus p)$;
    \end{enumerate}
    \item[(Add)] the generators are homomorphisms on the first argument;
    \item[(Comm)] \begin{enumerate}
        \item $[Z_{ij}(a, p), Z_{kl}(b, q)] = 1$ for distinct $\pm i$, $\pm j$, $\pm k$, $\pm l$,
        \item $[Z_{ij}(a, p), Z_{-i, j}(b, q)] = 1$,
        \item $\up{Z_{ij}(a, p)}{Z_{i \oplus j, k}(B, Q)} = Z_{i \oplus j, k}(\up{Z_{12}(a, p)} B, \up{Z_{12}(a, p)} Q)$;
    \end{enumerate}
    \item[(Simp)] $Z_{ik}(a, p) = Z_{i \oplus j, k}(a \oplus 0, p \oplus 0)$
    \item[(HW)] $Z_{j \oplus k, i}(\up{T_{12}(r)}{(0 \oplus a)}, p \oplus q) = Z_{k, i \oplus j}(\up{T_{12}(p)}{(a \oplus 0)}, \up{T_{12}(p)}{(q \oplus r)})$;
    \item[(Delta)] $\up{Z_{ji}(b, 0)}{Z_{ij}(a, p)} = Z_{ij}(a, p + b)$.
\end{itemize}

\section{Doubly laced Steinberg groups} \label{pairs-type}

In this and the next section \(\Phi\) is one of the root systems \(\mathsf B_\ell\), \(\mathsf C_\ell\), \(\mathsf F_4\). In order to define relative Steinberg groups of type \(\Phi\) over commutative rings with respect to Abe's admissible pairs, it is useful to consider Steinberg groups of type \(\Phi\) over pairs \((K, L)\), where \(K\) parametrizes the short root elements and \(L\) parametrizes the long root ones.

Let us say that \((K, L)\) is a \textit{pair of type} \(\mathsf B\) if
\begin{itemize}
\item \(L\) is a unital commutative ring, \(K\) is an \(L\)-module;
\item there is a classical quadratic form \(s \colon K \to L\), i.e. \(s(kl) = s(k) l^2\) and the expression \(s(k \mid k') = s(k + k') - s(k) - s(k')\) is \(L\)-bilinear.
\end{itemize}

Next, \((K, L)\) is a \textit{pair of type} \(\mathsf C\) if
\begin{itemize}
\item \(K\) is a unital commutative ring;
\item \(L\) is an abelian group;
\item there are additive maps \(d \colon K \to L\) and \(u \colon L \to K\);
\item there is a map \(L \times K \to L,\, (l, k) \mapsto l \cdot k\);
\item \((l + l') \cdot k = l \cdot k + l' \cdot k\), \(l \cdot (k + k') = l \cdot k + d(kk' u(l)) + l \cdot k'\);
\item \(u(d(k)) = 2k\), \(u(l \cdot k) = u(l) k^2\), \(d(u(l)) = 2l\), \(d(k) \cdot k' = d(k{k'}^2)\);
\item \(l \cdot 1 = l\), \((l \cdot k) \cdot k' = l \cdot kk'\).
\end{itemize}

Finally, \((K, L)\) is a \textit{pair of type} \(\mathsf F\) if
\begin{itemize}
\item \(K\) and \(L\) are unital commutative rings;
\item there is a unital ring homomorphism \(u \colon L \to K\);
\item there are maps \(d \colon K \to L\) and \(s \colon K \to L\);
\item \(d(k + k') = d(k) + d(k')\), \(d(u(l)) = 2l\), \(u(d(k)) = 2k\), \(d(k u(l)) = d(k) l\);
\item \(s(k + k') = s(k) + d(kk') + s(k')\), \(s(kk') = s(k) s(k')\), \(s(u(l)) = l^2\), \(u(s(k)) = k^2\).
\end{itemize}

If \((K, L)\) is a pair of type \(\mathsf C\) or \(\mathsf F\), we have an action \(K \times L \to K,\, (k, l) \mapsto kl = k u(l)\) and a biadditive map \(K \times K \to L, (k, k') \mapsto s(k \mid k') = d(kk')\). If \((K, L)\) is a pair of type \(\mathsf B\) or \(\mathsf F\), then there is a map \(L \times K \to L,\, (l, k) \mapsto l \cdot k = l s(k)\). With respect to these additional operations any pair \((K, L)\) of type \(\mathsf F\) is also a pair of both types \(\mathsf B\) and \(\mathsf C\). For any unital commutative ring \(K\) the pair \((K, K)\) with \(u(k) = k\), \(d(k) = 2k\), \(s(k) = k^2\) is of type \(\mathsf F\).

The \textit{Steinberg group} of type \(\Phi\) over a pair \((K, L)\) of the corresponding type is the abstract group \(\stlin(\Phi; K, L)\) with the generators \(x_\alpha(k)\) for short \(\alpha \in \Phi\), \(k \in K\); \(x_\beta(l)\) for long \(\beta \in \Phi\), \(l \in L\); and the relations
\begin{align*}
x_\alpha(p)\, x_\alpha(q) &= x_\alpha(p + q); \\
[x_\alpha(p), x_\beta(q)] &= 1 \text{ if } \alpha + \beta \notin \Phi \cup \{0\}; \\
[x_\alpha(p), x_\beta(q)] &= x_{\alpha + \beta}(N_{\alpha \beta} pq) \text{ if } \alpha + \beta \in \Phi \text{ and } |\alpha| = |\beta| = |\alpha + \beta|; \\
[x_\alpha(p), x_\beta(q)] &= \textstyle x_{\alpha + \beta}\bigl(\frac{N_{\alpha \beta}}2 s(p \mid q)\bigr) \text{ if } \alpha + \beta \in \Phi \text{ and } |\alpha| = |\beta| < |\alpha + \beta|; \\
[x_\alpha(p), x_\beta(q)] &= x_{\alpha + \beta}(N_{\alpha \beta} pq)\, x_{2\alpha + \beta}(N_{\alpha \beta}^{21} q \cdot p) \text{ if } \alpha + \beta, 2\alpha + \beta \in \Phi;\\
[x_\alpha(p), x_\beta(q)] &= x_{\alpha + \beta}(N_{\alpha \beta} qp)\, x_{\alpha + 2\beta}(N_{\alpha \beta}^{12} p \cdot q) \text{ if } \alpha + \beta, \alpha + 2\beta \in \Phi.
\end{align*}
Here \(N_{\alpha \beta}\), \(N_{\alpha \beta}^{21}\), \(N_{\alpha \beta}^{12}\) are the ordinary integer structure constants. In the case of the pair \((K, K)\) there is a canonical homomorphism
\[\stmap \colon \stlin(\Phi, K) = \stlin(\Phi; K, K) \to \group^{\mathrm{sc}}(\Phi, K)\]
to the simply connected Chevalley group over \(K\) of type \(\Phi\).

In order to apply the results on odd unitary groups, we need a construction of odd form rings by pairs of type \(\mathsf B\) and \(\mathsf C\). If \((K, L)\) is a pair of type \(\mathsf B\) and \(\ell \geq 0\), then we consider the special odd form ring \((R, \Delta) = \ofaorth(2\ell + 1; K, L)\) with
\begin{itemize}
\item \(R = (K \otimes_L K) e_{00} \oplus \bigoplus_{1 \leq |i| \leq \ell} (K e_{i0} \oplus K e_{0i}) \oplus \bigoplus_{1 \leq |i|, |j| \leq \ell} L e_{ij}\);
\item \(\inv{x e_{ij}} = x e_{-j, -i}\) for \(i \neq 0\) or \(j \neq 0\), \(\inv{(x \otimes y) e_{00}} = (y \otimes x) e_{00}\);
\item \((x e_{ij}) (y e_{kl}) = 0\) for \(j \neq k\);
\item \((x e_{ij}) (y e_{jk}) = xy e_{ik}\) for \(j \neq 0\) if at least one of \(i\) and \(k\) is non-zero;
\item \((x e_{0j}) (y e_{j0}) = (x \otimes y) e_{00}\) for \(j \neq 0\);
\item \((x e_{i0}) (y e_{0j}) = s(x \mid y) e_{ij}\) for \(i, j \neq 0\);
\item \((x e_{i0}) ((y \otimes z) e_{00}) = s(x \mid y) z e_{i0}\) for \(i \neq 0\);
\item \(((x \otimes y) e_{00}) ((z \otimes w) e_{00}) = (x \otimes s(y \mid z) w) e_{00}\);
\item \(\Delta\) is the subgroup of \(\Heis(R)\) generated by \(\phi(S)\), \(((x \otimes y) e_{00}, -(s(x) y \otimes y) e_{00})\), \((x e_{0i}, -s(x) e_{-i, i})\), \((x e_{i0}, 0)\), and \((x e_{ij}, 0)\) for \(i, j \neq 0\).
\end{itemize}
Clearly, \((R, \Delta)\) had a strong orthogonal hyperbolic family of rank \(\ell\) and the corresponding unitary Steinberg group is naturally isomorphic to \(\stlin(\mathsf B_\ell; K, L)\). The Steinberg relations in \(\stlin(\mathsf B_\ell; K, L)\) and \(\stunit(R, \Delta)\) are the same since \(\ofaorth(2\ell + 1, K, K)\) has the unitary group \(\sorth(2\ell + 1, K) \times (\mathbb Z / 2 \mathbb Z)(K)\) by \cite{ClassicOFA} for every unital commutative ring \(K\). Of course, \(\stlin(\mathsf B_\ell; K, L)\) may also be constructed by the module \(L^\ell \oplus K \oplus L^\ell\) with the quadratic form \(q(x_{-\ell}, \ldots, x_{-1}, k, x_1, \ldots, x_\ell) = \sum_{i = 1}^\ell x_{-i} x_i + s(k)\), but the corresponding odd form rings and unitary groups are not functorial on \((K, L)\).

If \((K, L)\) is a pair of the type \(\mathsf C\), then let \((R, \Delta) = \ofasymp(2\ell; K, L)\), where
\begin{itemize}
\item \(R = \bigoplus_{1 \leq |i|, |j| \leq \ell} K e_{ij}\);
\item \(\inv{x e_{ij}} = \varepsilon_i \varepsilon_j x e_{-j, -i}\), \((x e_{ij}) (y e_{kl}) = 0\) for \(j \neq k\), \((x e_{ij}) (y e_{jl}) = xy e_{il}\), where \(\varepsilon_i = 1\) for \(i > 0\) and \(\varepsilon_i = -1\) for \(i < 0\);
\item \(\Delta = \bigoplus^\cdot_{1 \leq |i|, |j| \leq \ell; i + j > 0} \phi(K e_{ij}) \dotoplus \bigoplus_{1 \leq |i| \leq \ell}^\cdot L v_i \dotoplus \bigoplus_{1 \leq |i|, |j| \leq \ell}^\cdot q_i \cdot K e_{ij}\);
\item \(x v_i \dotplus y v_i = (x + y) v_i\), \(q_i \cdot x e_{ij} \dotplus q_i \cdot y e_{ij} = q_i \cdot (x + y) e_{ij}\);
\item \(\phi(x e_{-i, i}) = d(x) v_i\), \(\pi(x v_i) = 0\), \(\rho(x v_i) = u(x) e_{-i, i}\);
\item \((x v_i) \cdot (y e_{jk}) = \dot 0\) for \(i \neq j\), \((x v_i) \cdot (y e_{ik}) = \varepsilon_i \varepsilon_k (x \cdot y) v_k\);
\item \(\pi(q_i \cdot x e_{ij}) = x e_{ij}\), \(\rho(q_i \cdot x e_{ij}) = 0\);
\item \((q_i \cdot x e_{ij}) \cdot (y e_{kl}) = \dot 0\) for \(j \neq k\), \((q_i \cdot x e_{ij}) \cdot (y e_{jk}) = q_i \cdot xy e_{ik}\);
\end{itemize}
Again, \((R, \Delta)\) had a strong orthogonal hyperbolic family of rank \(\ell\) and its unitary Steinberg group is naturally isomorphic to \(\stlin(\mathsf C_\ell; K, L)\). The Steinberg relations in these two Steinberg groups coincide since \(\ofasymp(2\ell; K, K)\) is the odd form ring constructed by the split symplectic module over a unital commutative ring \(K\), so its unitary group is \(\symp(2\ell, K)\).

We do not construct an analogue of \(\mathrm G(\mathsf F_4, K)\) for pairs of type \(\mathsf F\) and do not prove that the product map
\[\prod_{\text{short } \alpha \in \Pi} K \times \prod_{\text{long } \beta \in \Pi} L \to \stlin(\mathsf F_4; K, L), (x_\alpha)_{\alpha \in \Pi} \mapsto \prod_{\alpha \in \Pi} X_\alpha(p_\alpha)\]
is injective for a system of positive roots \(\Pi \subseteq \Phi\). Such claims are not required in the proof of our main result.

\section{Relative doubly laced Steinberg groups}

In the simply-laced case \cite{RelStLin} relative Steinberg groups are parameterized by the root system and a \textit{crossed module of commutative rings} \(\delta \colon \mathfrak a \to K\), where \(K\) is a unital commutative ring, \(\mathfrak a\) is a \(K\)-module, \(\delta\) is a homomorphism of \(K\)-modules, and \(a \delta(a') = \delta(a) a'\) for all \(a, a' \in \mathfrak a\). In the doubly-laced case we may construct semi-abelian categories of pairs of all three types by omitting the unitality conditions in the definitions, but in this approach we have to add the condition that the action in the definition of a crossed module is unital.

Instead we say that \((\mathfrak a, \mathfrak b)\) is a \textit{precrossed module} over a pair \((K, L)\) of a given type if there is a reflexive graph \(((K, L), (K', L'), p_1, p_2, d)\) in the category of pairs of a given type, where \((\mathfrak a, \mathfrak b) = \Ker(p_2)\). This may be written as explicit family of operations between the sets \(\mathfrak a\), \(\mathfrak b\), \(K\), \(L\) satisfying certain axioms, in particular, \(\delta \colon (\mathfrak a, \mathfrak b) \to (K, L)\) is a pair of homomorphisms of abelian groups induced by \(p_1\). A precrossed module \(\delta \colon (\mathfrak a, \mathfrak b) \to (K, L)\) is called a \textit{crossed module} if the corresponding reflexive graph has a structure of internal category (necessarily unique), this may be described as additional axioms on the operations (an analogue of Peiffer identities). It is easy to see that crossed submodules of \(\id \colon (K, L) \to (K, L)\) are precisely Abe's admissible pairs.

For a crossed module \(\delta \colon (\mathfrak a, \mathfrak b) \to (K, L)\) of pairs of a given type the \textit{relative Steinberg group} is
\[\stlin(\Phi; K, L; \mathfrak a, \mathfrak b) = \Ker(p_{2*}) / [\Ker(p_{1*}), \Ker(p_{2*})],\]
where \(p_{i*} \colon \stlin(\Phi; \mathfrak a \rtimes K, \mathfrak b \rtimes L) \to \stlin(\Phi; K, L)\) are the induced homomorphisms. As in the odd unitary case and the simply laced case, this is the crossed module over \(\stlin(\Phi; K, L)\) with the generators \(x_\alpha(a)\) for short \(\alpha \in \Phi\), \(a \in \mathfrak a\) and \(x_\beta(b)\) for long \(\beta \in \Phi\), \(b \in \mathfrak b\) satisfying the Steinberg relations, \(\delta(x_\alpha(a)) = x_\alpha(\delta(a))\) for any root \(\alpha\) and \(a \in \mathfrak a \cup \mathfrak b\), and
\[\up{x_\alpha(p)}{x_\beta(a)} = \prod_{\substack{i \alpha + j \beta \in \Phi \\ i \geq 0, j > 0}} x_{i \alpha + j \beta}(f_{\alpha \beta i j}(p, a))\]
for \(\alpha \neq -\beta\), \(p \in K \cup L\), \(a \in \mathfrak a \cup \mathfrak b\), and appropriate expressions \(f_{\alpha \beta i j}\).

If \(\delta \colon (\mathfrak a, \mathfrak b) \to (K, L)\) is a crossed module of pairs of type \(\mathsf C\) and \(\ell \geq 0\), then \(\delta \colon \ofasymp(\ell; \mathfrak a, \mathfrak b) \to \ofasymp(\ell; K, L)\) is a crossed module of odd form rings, where
\[\ofasymp(\ell; \mathfrak a, \mathfrak b) = \Ker\bigl(p_2 \colon \ofasymp(\ell; \mathfrak a \rtimes K, \mathfrak b \rtimes L) \to \ofasymp(\ell; K, L)\bigr).\]
Clearly,
\[\ofasymp(\ell; \mathfrak a, \mathfrak b) = \Bigl( \bigoplus_{1 \leq |i|, |j| \leq \ell} \mathfrak a e_{ij},
\bigoplus^\cdot_{\substack{1 \leq |i|, |j| \leq \ell \\ i + j > 0}} \phi(\mathfrak a e_{ij}) \dotoplus \bigoplus^\cdot_{1 \leq |i| \leq \ell} \mathfrak b v_i \dotoplus \bigoplus^\cdot_{1 \leq |i|, |j| \leq \ell} q_i \cdot \mathfrak a e_{ij} \Bigr),\]
so we may apply theorem \ref{pres-stu} for Chevalley groups of type \(\mathsf C_\ell\).

For pairs of type \(\mathsf B\) the construction \(\ofaorth(2\ell + 1; -, =)\) does not preserve fiber products, so we have to modify it a bit. Take a crossed module \(\delta \colon (\mathfrak a, \mathfrak b) \to (K, L)\) of pairs of type \(\mathsf B\) and consider the odd form rings \((T, \Xi) = \ofaorth(2\ell + 1; \mathfrak a \rtimes K, \mathfrak b \rtimes L)\), \((R, \Delta) = \ofaorth(2\ell + 1; K, L)\) forming a reflexive graph. Let \((\widetilde T, \widetilde \Xi)\) be the special odd form factor ring of \((T, \Xi)\) by the odd form ideal \((I, \Gamma)\), where \(I \leq T\) is the subgroup generated by \((a \otimes a' - a \otimes \delta(a')) e_{00}\) and \((a \otimes a' - \delta(a) \otimes a') e_{00}\) for \(a, a' \in \mathfrak a\). It is easy to check that \(I\) is an involution invariant ideal of \(T\), so \((\widetilde T, \widetilde \Xi)\) is well-defined, and the homomorphisms \(p_i \colon (T, \Xi) \to (R, \Delta)\) factor through \((\widetilde T, \widetilde \Xi)\). Moreover, the precrossed module \((S, \Theta) = \Ker\bigl(p_2 \colon (\widetilde T, \widetilde \Xi) \to (R, \Delta)\bigr)\) over \((R, \Delta)\) satisfies the Peiffer relations and
\[S = X e_{00} \oplus \bigoplus_{1 \leq |i| \leq \ell} (\mathfrak a e_{i0} \oplus \mathfrak a e_{0i}) \oplus \bigoplus_{1 \leq |i|, |j| \leq \ell} \mathfrak b e_{ij}\]
for some group \(X\), so we may also apply theorem \ref{pres-stu} for Chevalley groups of type \(\mathsf B_\ell\).

If \(\alpha\), \(\beta\), \(\alpha - \beta\) are short roots, then \(V_{\alpha \beta}(\mathfrak a, \mathfrak b)\) denotes the abelian group \(\mathfrak a \mathrm e_\alpha \oplus \mathfrak a \mathrm e_\beta\). The groups \(X_{\alpha - \beta}(K)\) and \(X_{\beta - \alpha}(K)\) naturally act on \(V_{\alpha \beta}(\mathfrak a, \mathfrak b)\) in such a way that the homomorphism
\[V_{\alpha \beta}(\mathfrak a, \mathfrak b) \to \stlin(\Phi; \mathfrak a, \mathfrak b), x \mathrm e_\alpha \oplus y \mathrm e_\beta \mapsto X_\alpha(x)\, X_\beta(y)\]
is equivariant (in the case of \(\mathsf F_4\) we consider \(X_{\pm(\alpha - \beta)}(K)\) as abstract groups, not as their images in the Steinberg group). Similarly, if \(\alpha\), \(\beta\), \(\alpha - \beta\) are long roots, then \(V_{\alpha \beta}(\mathfrak a, \mathfrak b) = \mathfrak b \mathrm e_\alpha \oplus \mathfrak b \mathrm e_\beta\) is a representation of \(X_{\alpha - \beta}(L)\) and \(X_{\beta - \alpha}(L)\). If \(\alpha\) and \(\beta\) are long and \((\alpha + \beta)/2\) is short, then \(V_{\alpha \beta}(\mathfrak a, \mathfrak b) = \mathfrak b \mathrm e_\alpha \oplus \mathfrak a \mathrm e_{(\alpha + \beta)/2} \oplus \mathfrak b \mathrm e_\beta\) is a representation of \(X_{(\alpha - \beta)/2}(K)\) and \(X_{(\beta - \alpha)/2}(K)\). Finally, if \(\alpha\) and \(\beta\) are short and \(\alpha + \beta\) is long, then \(V_{\alpha \beta}(\mathfrak a, \mathfrak b) = \mathfrak a \mathrm e_\alpha \dotoplus \mathfrak b \mathrm e_{\alpha + \beta} \dotoplus \mathfrak a \mathrm e_\alpha\) is a \(2\)-step nilpotent group with the group operation
\[\textstyle (x \mathrm e_\alpha \dotoplus y \mathrm e_{\alpha + \beta} \dotoplus z \mathrm e_\beta) \dotplus (x' \mathrm e_\alpha \dotoplus y' \mathrm e_{\alpha + \beta} \dotoplus z' \mathrm e_\beta) = (x + x') \mathrm e_\alpha \dotoplus \bigl(y - \frac{N_{\alpha \beta}}2 s(z \mid x') + z'\bigr) \mathrm e_{\alpha + \beta} \dotoplus (z + z') \mathrm e_\beta\]
and the action of \(X_{\alpha - \beta}(L)\) and \(X_{\beta - \alpha}(L)\) such that the homomorphism
\[V_{\alpha \beta}(\mathfrak a, \mathfrak b) \to \stlin(\Phi; \mathfrak a, \mathfrak b), x \mathrm e_\alpha \dotoplus y \mathrm e_{\alpha + \beta} \dotoplus z \mathrm e_\beta \mapsto X_\alpha(x)\, X_{\alpha + \beta}(y)\, X_\beta(z)\]
is equivariant.

We are ready to construct a presentation of \(G = \stlin(\Phi; K, L; \mathfrak a, \mathfrak b)\). Let \(Z_\alpha(x, p) = \up{X_{-\alpha}(p)}{X_\alpha(x)} \in G\) for \(x \in \mathfrak a\), \(p \in K\) if \(\alpha\) is short and \(x \in \mathfrak b\), \(p \in L\) if \(\alpha\) is long. If \(V_{\alpha \beta}(K, L)\) is defined, then there are also natural elements \(Z_{\alpha \beta}(u, s) \in G\) for \(u \in V_{\alpha \beta}(\mathfrak a, \mathfrak b)\) and \(s \in V_{-\alpha, -\beta}(K, L)\)

\begin{theorem}\label{pres-stphi}
Let \(\Phi\) be one of the root systems \(\mathsf B_\ell\), \(\mathsf C_\ell\), or \(\mathsf F_4\) for \(\ell \geq 3\); \((K, L)\) be a pair of the corresponding type; \((\mathfrak a, \mathfrak b)\) be a crossed module over \((K, L)\). Then \(\stlin(\Phi; K, L; \mathfrak a, \mathfrak b)\) as an abstract group is generated by \(Z_\alpha(x, p)\) and \(Z_{\alpha \beta}(u, s)\) with the relations
\begin{itemize}
\item[(Sym)]
 \(Z_{\alpha \beta}(u, s) = Z_{\beta \alpha}(u, s)\) if we identify \(V_{\alpha \beta}(\mathfrak a, \mathfrak b)\) with \(V_{\beta \alpha}(\mathfrak a, \mathfrak b)\);
\item[(Add)]
 \begin{enumerate}
  \item \(Z_\alpha(x, p)\, Z_\alpha(y, p) = Z_\alpha(x + y, p)\),
  \item \(Z_{\alpha \beta}(u, s)\, Z_{\alpha \beta}(v, s) = Z_{\alpha \beta}(u \dotplus v, s)\);
 \end{enumerate}
\item[(Comm)]
 \begin{enumerate}
 \item \([Z_\alpha(x, p), Z_\beta(y, q)] = 1\) for \(\alpha \perp \beta\) and \(\alpha + \beta \notin \Phi\),
 \item \(\up{Z_\gamma(x, p)}{Z_{\alpha \beta}(u, s)} = Z_{\alpha \beta}\bigl(Z_\gamma(\delta(x), p) u, Z_\gamma(\delta(x), p) s\bigr)\) for \(\gamma \in \mathbb R (\alpha - \beta)\);
 \end{enumerate}
\item[(Simp)] \(Z_\alpha(x, p) = Z_{\alpha \beta}(x \mathrm e_\alpha, p \mathrm e_{-\alpha})\);
\item[(HW)]
 \begin{enumerate}
 \item \(Z_{\alpha, \alpha + \beta}\bigl(X_{-\beta}(r)\, x \mathrm e_{\alpha + \beta}, u\bigr) = Z_{\alpha + \beta, \beta}\bigl(X_{-\alpha}(p)\, x \mathrm e_{\alpha + \beta}, X_{-\alpha}(p) v\bigr)\), where \(\alpha\), \(\beta\) is a basis of root subsystem of type \(\mathsf A_2\), \(u = p \mathrm e_{-\alpha} \oplus q \mathrm e_{-\alpha - \beta}\), \(v = q \mathrm e_{-\alpha - \beta} \oplus r \mathrm e_{-\beta}\),
 \item \(Z_{\alpha, \alpha + \beta}\bigl(X_{-\beta}(s) u, v\bigr) = Z_{2 \alpha + \beta, \beta}\bigl(X_{-\alpha}(p) u, X_{-\alpha}(p) w\bigr)\), where \(\alpha\), \(\beta\) is a basis of a root subsystem of type \(\mathsf B_2\) and \(\alpha\) is short, \(u = x \mathrm e_{2\alpha + \beta} \dotoplus y \mathrm e_{\alpha + \beta}\), \(v = p \mathrm e_{-\alpha} \dotoplus q \mathrm e_{-2 \alpha - \beta} \dotoplus r \mathrm e_{-\alpha - \beta}\), \(w = q \mathrm e_{-2\alpha - \beta} \oplus r \mathrm e_{-\alpha - \beta} \oplus s \mathrm e_{-\beta}\);
 \end{enumerate}
\item[(Delta)] \(Z_\alpha(x, \delta(y) + p) = \up{Z_{-\alpha}(y, 0)}{Z_\alpha(x, p)}\).
\end{itemize}
\end{theorem}
\begin{proof}
If \(\Phi\) is of type \(\mathsf B_\ell\) or \(\mathsf C_\ell\), then the claim directly follows from theorem \ref{pres-stu}, so we may assume that \(\Phi\) is of type \(\mathsf F_4\). Let \(G\) be the group with the presentation from the statement, it is generated by the elements \(Z_\alpha(a, p)\) satisfying only the relations involving roots from root subsystems of rank \(2\). First of all, we construct a natural action of \(\stlin(\Phi; K, L)\) on \(G\). Notice that any three-dimensional root subsystem \(\Psi \subseteq \Phi\) such that \(\Psi = \mathbb R \Psi \cap \Phi\) is of type \(\mathsf B_3\), \(\mathsf C_3\), or \(\mathsf A_1 \times \mathsf A_2\).

Let \(g = X_\alpha(a)\) be a root element in \(\stlin(\Phi; K, L)\) and \(h = Z_\beta(b, p)\) be a generator of \(G\). If \(\alpha\) and \(\beta\) are linearly independent, then \(\up gh\) may be defined directly as \(h\) itself or an appropriate \(Z_{\gamma_1 \gamma_2}(u, s)\). Otherwise we take a root subsystem \(\alpha \in \Psi \subseteq \Phi\) of rank \(2\) of type \(\mathsf A_2\) or \(\mathsf B_2\), express \(h\) in terms of \(Z_\gamma(c, q)\) for \(\gamma \in \Psi \setminus \mathbb R \alpha\), and apply the above construction for \(\up g{Z_\gamma(c, q)}\). The resulting element of \(G\) is independent of the choices of \(\Psi\) and the decomposition of \(h\) since we already know the theorem in the cases \(\mathsf B_3\) and \(\mathsf C_3\).

We have to check that \(\up g{(-)}\) preserves the relations between the generators. Let \(\Psi\) be the intersection of \(\Phi\) with the span of the roots in a relation, \(\Psi' = (\mathbb R \Psi + \mathbb R \alpha) \cap \Phi\). Consider the possible cases:
\begin{itemize}
\item If \(\Psi'\) is of rank \(< 3\) or has one of the types \(\mathsf B_3\), \(\mathsf C_3\), then the result follows from the cases \(\mathsf B_3\) and \(\mathsf C_3\) of the theorem.
\item If \(\Psi'\) is of type \(\mathsf A_1 \times \mathsf A_2\) and \(\alpha\) lies in the first component (so \(\Psi\) is the second component), then \(g\) acts trivially on the generators with the roots from \(\Psi\) and there is nothing to prove.
\item If \(\Psi'\) is of type \(\mathsf A_1 \times \mathsf A_2\), \(\alpha\) lies in the second component, and \(\Psi\) is of type \(\mathsf A_1 \times \mathsf A_1\), then we have to check that \(\bigl[\up g{Z_\beta(b, p)}, \up g{Z_\gamma(c, q)}\bigr] = 1\), where \(\beta\) lies in the first component and \(\gamma\) lies in the second component. But \(\up g{Z_\beta(b, p)} = Z_\beta(b, p)\) commutes with \(\up g{Z_\gamma(c, q)}\), the latter one is a product of various \(Z_{\gamma'}(c', q')\) with \(\gamma'\) from the second component of \(\Psi'\).
\end{itemize}

Now let us check that the resulting automorphisms of \(G\) corresponding to root elements satisfy the Steinberg relations when applied to a fixed \(Z_\beta(b, p)\). Let \(\Psi\) be the intersection of \(\Phi\) with the span of the root from such a relation and \(\Psi' = (\mathbb R \Psi + \mathbb R \beta) \cap \Phi\). There are the following cases:
\begin{itemize}
\item If \(\Psi'\) is of rank \(< 3\) or has one of the types \(\mathsf B_3\), \(\mathsf C_3\), then the result follows from the cases \(\mathsf B_3\) and \(\mathsf C_3\) of the theorem.
\item If \(\Psi'\) is of type \(\mathsf A_1 \times \mathsf A_2\) and \(\beta\) lies in the first component (so \(\Psi\) is the second component), then both sides of the relation trivially act on \(Z_\beta(b, p)\) and there is nothing to prove.
\item If \(\Psi'\) is of type \(\mathsf A_1 \times \mathsf A_2\), \(\beta\) lies in the second component, and \(\Psi\) is of type \(\mathsf A_1 \times \mathsf A_1\), then we have to check that \(\up{X_\alpha(a)\, X_\gamma(c)}{Z_\beta(b, p)} = \up{X_\gamma(c)\, X_\alpha(a)}{Z_\beta(b, p)}\), where \(\alpha\) lies in the first component and \(\gamma\) lies in the second component. But both sides coincide with \(\up{X_\gamma(c)}{Z_\beta(b, p)}\), it is a product of various \(Z_{\beta'}(b', p')\) with \(\beta'\) from the second component of \(\Psi'\).
\end{itemize}

Now consider the homomorphism
\[u \colon G \to \stlin(\Phi; K, L; \mathfrak a, \mathfrak b), Z_\alpha(a, p) \mapsto \up{X_{-\alpha}(p)}{X_\alpha(a)}.\]
By construction, it is \(\stlin(\Phi; K, L)\)-equivariant, so it is surjective. The \(\stlin(\Phi; K, L)\)-equivariant homomorphism
\[v \colon \stlin(\Phi; K, L; \mathfrak a, \mathfrak b) \to G, X_\alpha(a) \mapsto X_\alpha(a)\]
is clearly well-defined. It remains to notice that \(v \circ u\) is the identity.
\end{proof}

\bibliographystyle{plain}
\bibliography{references}

\end{document}